\documentclass{eptcs}

\usepackage{amsthm}
\usepackage{mathtools}

\usepackage[inline]{enumitem}
\usepackage{ifthen}

\usepackage{xcolor}

\usepackage{tikz}

\usepackage{amssymb}
\usepackage[varg,bigdelims]{newpxmath}
\usepackage{mathrsfs}
\usepackage{dutchcal}
\usepackage{fontawesome}

\usepackage{lineno}
\usepackage[capitalize]{cleveref}
\usepackage{mathtools}
\usepackage{caption}
\usepackage{subcaption}

\definecolor{dgreen}{rgb}{0.0, 0.5, 0.3} 
\definecolor{dyellow}{rgb}{8.0, 0.74, 0}

  \hypersetup{final}

  \setlist{nosep}
  \setlistdepth{6}

\usepackage{thmtools, thm-restate}
\usepackage{amsthm}

\theoremstyle{definition}
\newtheorem{definitionx}{Definition}[section]
\newtheorem{examplex}[definitionx]{Example}
\newtheorem{remarkx}[definitionx]{Remark}

\theoremstyle{plain}

\newtheorem{proposition}[definitionx]{Proposition}
\newtheorem{corollary}[definitionx]{Corollary}
\newtheorem{lemma}[definitionx]{Lemma}

\newenvironment{example}
  {\pushQED{\qed}\examplex}
  {\popQED\endexamplex}
  
 \newenvironment{remark}
  {\pushQED{\qed}\remarkx}
  {\popQED\endremarkx}
  
  \newenvironment{definition}
  {\pushQED{\qed}\definitionx}
  {\popQED\enddefinitionx}

\DeclareSymbolFont{stmry}{U}{stmry}{m}{n}
\DeclareMathSymbol\fatsemi\mathop{stmry}{"23}

\DeclareMathOperator*{\colim}{colim}

\DeclareMathOperator{\ob}{Ob}

\DeclareMathOperator{\el}{El}
\DeclareMathOperator*{\argmin}{arg\,min}

\newcommand{\Set}[1]{\mathrm{#1}}
\newcommand{\cat}[1]{\mathcal{#1}}
\newcommand{\Cat}[1]{{\mathbf{#1}}}
\newcommand{\fun}[1]{\mathrm{#1}}

\newcommand{\id}{\mathrm{id}}
\newcommand{\then}{\mathbin{\fatsemi}}

\newcommand{\iso}{\cong}
\newcommand{\too}{\longrightarrow}

\newcommand{\To}[2][]{\xrightarrow[#1]{#2}}

\newcommand{\inv}{^{-1}}
\newcommand{\op}{^\tn{op}}

\newcommand{\tn}[1]{\textnormal{#1}}
\newcommand{\ol}[1]{\overline{#1}}

\newcommand{\nn}{\mathbb{N}}

\newcommand{\ja}{\mathcal{J}}

\newcommand{\smset}{\Cat{Set}}
\newcommand{\smcat}{\Cat{Cat}}
\newcommand{\catsharp}{\Cat{Cat}^{\sharp}}
\newcommand{\oorg}{\mathbb{O}\Cat{rg}}

\newcommand{\polycart}{\poly_{\Cat{Cart}}}
\newcommand{\Mod}{\Cat{Mod}}

\newcommand{\yon}{\mathcal{y}}
\newcommand{\poly}{\Cat{Poly}}
\newcommand{\tri}{\mathbin{\triangleleft}}

\newcommand{\free}{\mathfrak{m}}
\newcommand{\cofree}{\mathfrak{c}}

\newcommand{\ext}{\fun{Ext}}

\newcommand{\Bool}{\Set{Bool}}

\newcommand{\biglens}[2]{
     \begin{bmatrix}{\vphantom{f_f^f}#2} \\ {\vphantom{f_f^f}#1} \end{bmatrix}
}
\newcommand{\littlelens}[2]{
     \begin{bsmallmatrix}{\vphantom{f}#2} \\ {\vphantom{f}#1} \end{bsmallmatrix}
}
\newcommand{\lens}[2]{
  \relax\if@display
     \biglens{#1}{#2}
  \else
     \littlelens{#1}{#2}
  \fi
}

\newcommand{\qqand}{\qquad\text{and}\qquad}

\newcommand{\coto}{\nrightarrow}

\newcommand{\coh}[1]{^{(#1)}}
\newcommand{\hoc}[1]{_{(#1)}}

\newcommand{\modstruct}{\Xi}

\newcommand{\freeu}[1]{\eta} 
\newcommand{\freem}[1]{\mu}  

  \usetikzlibrary{ 
  	cd,
  	math,
  	decorations.markings,
		decorations.pathreplacing,
  	positioning,
  	arrows.meta,
  	shapes,
		shadows,
		shadings,
  	calc,
  	fit,
  	quotes,
  	intersections,
    circuits,
    circuits.ee.IEC
  }
  
  \tikzset{
biml/.tip={Glyph[glyph math command=triangleleft, glyph length=.95ex]},
bimr/.tip={Glyph[glyph math command=triangleright, glyph length=.95ex]},
}

\tikzset{
	tick/.style={postaction={
  	decorate,
    decoration={markings, mark=at position 0.5 with
    	{\draw[-] (0,.4ex) -- (0,-.4ex);}}}
  }
} 
\tikzset{
	slash/.style={postaction={
  	decorate,
    decoration={markings, mark=at position 0.5 with
    	{\draw[-] (.3ex,.3ex) -- (-.3ex,-.3ex);}}}
  }
} 

\tikzset{trees/.style={
	inner sep=0, 
	minimum width=0, 
	minimum height=0,
	level distance=.75cm, 
	sibling distance=.5cm,
	edge from parent/.style={shorten <= 2pt, draw, ->},
	grow'=up,
	decoration={markings, mark=at position 0.75 with \arrow{stealth}}
	}
}

\newcommand{\adj}[5][30pt]{
\begin{tikzcd}[ampersand replacement=\&, column sep=#1]
  #2\ar[r, shift left=6pt, "#3"]
  \ar[r, phantom, "\scriptstyle\Rightarrow"]\&
  #5\ar[l, shift left=6pt, "#4"]
\end{tikzcd}
}

\usetikzlibrary{automata, positioning, arrows}

\usepackage[draft]{minted} 
\usepackage{quiver}

\usepackage[utf8]{inputenc} 

\title{Pattern Runs on Matter: The Free Monad Monad as a Module over the Cofree Comonad Comonad}

\author{
    Sophie Libkind
    \institute{Topos Institute\\
    Berkeley, CA, USA}
    \email{sophie@topos.institute}
\and
    David I. Spivak
    \institute{Topos Institute\\
    Berkeley, CA, USA}
    \email{david@topos.institute}
}

\def\titlerunning{Pattern Runs on Matter}
\def\authorrunning{S.\ Libkind and D.\ I.\ Spivak}

\begin{document}

\maketitle

\begin{abstract}
Interviews run on people, programs run on operating systems, voting schemes run on voters, games run on players. Each of these is an example of the abstraction \textit{pattern runs on matter}. Pattern determines the decision tree that governs how a situation can unfold, while matter responds with decisions at each juncture. 

In this article, we will give a straightforward and concrete construction of the free monad monad for $(\poly, \tri, \yon)$, the category of polynomial functors with the substitution monoidal product. 
Although the free monad has been well-studied in other contexts, the construction we give is streamlined and explicitly illustrates how the free monad represents terminating decision trees. We will also explore the naturally arising interaction between the free monad and cofree comonad. Again, while the interaction itself is known, the perspective we take is the free monad as a module over the cofree comonad. Lastly, we will give four applications of the module action to interviews, computer programs, voting, and games. In each example, we will see how the free monad represents pattern, the cofree comonad represents matter, and the module action represents \emph{runs on}. 
\end{abstract}

\section{Introduction}\label{sec:intro}

The etymology of matter and pattern are ``mother'' and ``father''; this pair of terms offers a very basic sense in which to carve up the world. Like two parents, matter and pattern represent a fundamental dichotomy: matter is the pure material, unconcerned with our ideas about it; pattern is pure structure, unconcerned with what substantiates it. And yet one may have an intuitive sense that pattern ``runs on''---must be instantiated in---matter. In this paper, we show that this idea matches both our intuition and the mathematics of a module structure by which free monads are a module over cofree comonads. 

Intuitively, interviews, programs, voting schemes, and games represent patterns. But what is an interview without a person to be interviewed; what is a program without a operating system to run it on; what is a voting scheme without voters to do the voting; what is a game without players to play it? In each case, the pattern runs on a material substrate.

In this paper we give an account for this intuition in terms of free monads $\free$ and cofree comonads $\cofree$ on polynomial functors. We show that for any $p,q:\poly$, there is a natural map
\begin{equation}\label{eqn.module_struct}
	\modstruct_{p,q}\colon\free_p\otimes\cofree_q\to\free_{p\otimes q}.
\end{equation}
In fact, this map gives $\free$ the structure of a $\cofree$-module, in a precise sense; see \cref{sec.module}. The free monad $\free_p$ represents patterns as terminating decision tree with shape $p$,  the cofree comonad $\cofree_q$ represents matter as infinite  behaviors trees of shape $q$, and the interaction law $\modstruct_{p,q}$ interprets ``runs on''.

Consider the behavior of a Moore machine, which transforms lists of $A$-inputs into lists of $B$-outputs for sets $A,B$. This behavior is determined by an element $b\colon\yon\to\cofree_{B\yon^A}$ of the terminal coalgebra on the polynomial $q\coloneqq B\yon^A$. But how does one apply this behavior to a list of $A$'s? The latter is a map $\ell\colon\yon\to\free_{A\yon}$. To actually apply the behavior $b$ to the list $\ell$, we use an obvious map $\varphi\colon A\yon\otimes B\yon^A\to B\yon$ and the module structure \eqref{eqn.module_struct} to obtain a list of $B$'s.
\begin{equation}\label{eqn.list_to_list}
\yon\cong\yon\otimes\yon\To{\ell\otimes b}\free_{A\yon}\otimes\cofree_{B\yon^A}\To{\modstruct_{A\yon,B\yon^A}}\free_{A\yon\otimes B\yon^A}\To{\free_\varphi}\free_{B\yon}
\end{equation}

Whereas patterns start and end, matter is never destroyed. One can think of $\free_p$ as the type of terminating programs and of $\cofree_q$ as the type of operating systems or online algorithms. This intuition is captured by the fact that elements of the free monad $\free_p$ are ``well-founded trees'' \cite{abramsky2001handbook, kock2012polynomial}---in the case $p$ is finitary, a wellfounded tree is one with finite height---whereas elements of the cofree comonad $\cofree_q$ are generally non-wellfounded, e.g. infinite in height even for finitary~$q$. The module map $\modstruct_{p,q}$ pairs the wellfounded tree with the non-wellfounded tree, following the shape of the wellfounded one; for example the list of $B$'s in \eqref{eqn.list_to_list} will have exactly the same length as the list of $A$'s.

\paragraph{Related work.}
A construction of the free monad for much more general (than polynomial) endofunctors was given by Kelly in \cite{kelly1980unified}. The treatment was greatly simplified by Shulman and others in \cite{nlab:transfinite_construction_of_free_algebras}. The case of polynomial endofunctors is more restrictive, and as such has a far more straightforward construction.

A monad-comonad interaction law, again in more general settings, was described in \cite{Katsumata2020interaction}. This paper structures interaction laws in an interesting and useful but somehow ad hoc way, as a category whose objects are triples $(T,C,f)$, where $T$ is a monad, $C$ is a comonad, and $f\colon T\otimes C\to\id$ is a map satisfying certain natural conditions, and whose morphisms are maps $T\to T'$ and $C'\to C$ satisfying certain constraints.

\paragraph{Contributions.}

\begin{enumerate}
    \item Simple concrete constructions of both free monads and cofree comonads in $(\poly, \tri, \yon)$.
    \item A proof that the free monad is a module over the cofree comonad.
    \item Four applications of this module, each having the form \textit{pattern runs on matter}.
\end{enumerate}
\smallskip
For these, see \cref{def.free_monad,prop:cofree}; \cref{thm.main}; and \cref{sec:apps}, respectively. We also give a simplified definition of the  ``dual'' of a polynomial functor that specializes the notion given in \cite{Katsumata2020interaction}, and we generalize it to be functorial in a monad $t$; see \eqref{eqn.generalize_dual}.

\paragraph{Notation.}

We often denote the identity on an object $x$ by the object name itself rather than $\id_x$. We denote the cardinality of a set $X$ by $\#X$. 
If $L\dashv R$ is an adjunction, we denote it
\[
\adj{\cat{C}}{L}{R}{\cat{D}}
\]
so that the 2-cell shown indicates the direction of both the unit $\cat{C}\to R \circ L$ and the counit $L \circ R \to\cat{D}$.

\paragraph{Acknowledgements.}
The authors thanks Harrison Grodin and Brandon Shapiro for many useful conversations.
This material is based upon work supported by the Air Force Office of Scientific Research under award numbers FA9550-20-1-0348 and FA9550-23-1-0376.

\subsection{Preliminaries}

Although we will assume basic familiarity with the category $\poly$ of univariate polynomial functors on $\smset$ and natural transformations between them, we will begin by clarifying notation. We write a polynomial functor $p: \poly$ as $p = \sum_{P: p(1)} \yon^{p[P]}$ where $p(1)$ are the \textbf{positions} of $p$ and $p[P]$ are the \textbf{directions} at the position $P$. With this notation, a map $\phi: p \to q$ in $\poly$ consists of 
\begin{itemize}
    \item A function on positions $\phi(1): p(1) \to q(1)$.
    \item For each position $P: p(1)$, a function backwards on directions $\phi^\#_P: q[\phi_1(P)] \to p[P]$.
\end{itemize}
A map $\phi\colon p \to q$ in $\poly$ is \textbf{cartesian} if for every position $P: p(1)$, the map on directions $\phi^\#_P$ is an isomorphism (equivalently, when all $\phi$'s naturality squares are pullbacks). 

There are many monoidal products in $\poly$~\cite{niu2022poly}, however in this article we focus on the Dirichlet and substitution products. 
The substitution product has unit $\yon$ and is defined by
\[
    p \tri q = \sum_{P: p(1)} \prod_{p[P]}\sum_{Q:q(1)}\prod_{q[Q]} \yon
\]
Using the distributive law, a position in $p \tri q$ is a position $P$ in $p$ and for every direction in $p[P]$ a position in $q$.
It will be useful to consider the corolla forest view on polynomials and the $\tri$-monoidal product, since these will be the building blocks of the decision trees represented by the free monad monad. 

Consider $p = \yon^3 + \yon^2$. We can notate $p$ using a corolla forest in which each position $P:p(1)$ is a tree with $p[P]$-many branches. The root of each tree is labeled by its corresponding position and  branches correspond to  directions.
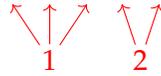
\[
\begin{tikzpicture}[trees,red, 
  level 1/.style={sibling distance=5mm},
  level 2/.style={sibling distance=2.5mm},
  level 3/.style={sibling distance=1.25mm}]
	\node (a) {$0$}
    	child {node {}} 
    	child {node {}}
    	child {node {}}
		;
        \node[right = of a] (b) {$1$}
		  child {node {}}
		  child {node {}}
		;
\end{tikzpicture}
\]%
Likewise, the polynomial $q = 2\yon^4 + \yon^2 + 1$ is notated by the corolla forest:
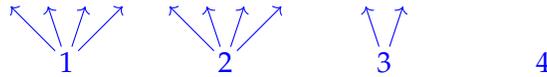
\[
\begin{tikzpicture}[trees,blue, 
  level 1/.style={sibling distance=5mm},
  level 2/.style={sibling distance=2.5mm},
  level 3/.style={sibling distance=1.25mm}]
	\node (a) {$0$}
    	child {node {}}
    	child {node {}}
            child {node {}}
            child {node {}}
		;
        \node[right = 0.75in of a] (b) {$1$}
		  child {node {}}
		  child {node {}}
            child {node {}}
            child {node {}}
		;
        \node[right = 0.75in of b] (c) {$2$}
    	child {node {}}
    	child {node {}}
		;
        \node[right = 0.75in of c] (d) {$3$}
		;
\end{tikzpicture}
\]%
\mbox{}\vspace{-.1in}\mbox{}

Then the positions of $p\tri q$ can be represented by trees of height 2, whose building blocks are the trees corresponding to the positions of $p$ and $q$. In particular, a position of $p \tri q$ is a position of $p$ and for each branch, a position of $q$. The directions are the number of dangling leaves. Here are three examples of positions in $p \tri q$ with eight, eight, and two directions respectively.

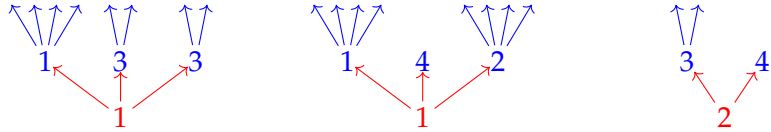
\[
\begin{tikzpicture}[trees,red, 
  level 1/.style={sibling distance=10mm},
  level 2/.style={sibling distance=3mm},
  level 3/.style={sibling distance=1.25mm}]
	\node[red] (a) {$0$}
    	child[red] {node[blue] {$0$}
                    child[blue]
                    child[blue]
                    child[blue]
                    child[blue]
            }
            child[red] {node[blue] {$2$}
                child[blue]
                child[blue]
            }
            child[red] {node[blue] {$2$}
                child[blue]
                child[blue]
            }
        ;

        \node[red, right = 1.5in of a] (b) {$0$}
    	child[red] {node[blue] {$0$}
                child[blue]
                child[blue]
                child[blue]
                child[blue]
            }
            child[red] {node[blue] {$3$}
            }
            child[red] {node[blue] {$1$}
                child[blue]
                child[blue]
                child[blue]
                child[blue]
            }
        ;
        \node[red, right = 1.5in of b] (c) {$1$}
    	child[red] {node[blue] {$2$}
                child[blue]
                child[blue]
            }
            child[red] {node[blue] {$3$}
            }
        ;
\end{tikzpicture}
\]%
\mbox{}\vspace{-.3in}\mbox{}

\section{The Free Monad Monad}\label{sec:free-monad}
For a polynomial $p$, we will define the free monad $\free_p$ through transfinite induction~\cite{nlab:transfinite_construction_of_free_algebras} by defining polynomials $p\hoc{\alpha}$ for each ordinal $\alpha$ and inclusions $p\hoc{\alpha} \to p \hoc{\beta}$ for ordinals $\alpha < \beta$.
\begin{itemize}
    \item Base case:  Define $p\hoc{0} \coloneqq \yon$.
    \item For successor ordinals $\alpha+1$: Define $p\hoc{\alpha + 1} \coloneqq \yon + p \tri p\hoc{\alpha}$. For the inclusions define $\iota\hoc{0}: \yon \to \yon + p$ to be the inclusion; note that it is cartesian. Then, given $\iota\hoc{\alpha}$, define $\iota\hoc{\alpha + 1}$ to be
    \[
        p\hoc{\alpha+1} = \yon + p \tri p \hoc{\alpha} \xrightarrow{\yon + p \tri \iota\hoc{\alpha}} \yon + p \tri p \hoc{\alpha +1} = p \hoc{\alpha + 2}.
    \]
    which is also cartesian.
    \item For limit ordinals $\alpha$: Define $p \hoc{\alpha} \coloneqq \colim_{\alpha' < \alpha} p\hoc{\alpha'}$.  For each $\alpha' < \alpha$, let $\iota\hoc{\alpha'} \colon p\hoc{\alpha'} \to p\hoc{\alpha}$ be the natural inclusion. These are cartesian by \cref{prop.polycart_colimits}. Lastly, we define the inclusion $\iota\hoc{\alpha} \colon p \hoc{\alpha} \to p \hoc{\alpha + 1}$ to be induced by the cocone of maps which are defined for $\alpha' < \alpha$ by
    \begin{equation}\label{eq:limit-to-limit-plus-one}
        p\hoc{\alpha'} \xrightarrow{\iota\hoc{\alpha'}} p \hoc{\alpha' + 1} = \yon + p \tri p \hoc{\alpha'} \xrightarrow{\yon + p \tri \iota\hoc{\alpha'}} \yon + p \tri p\hoc{\alpha} = p\hoc{\alpha + 1}.
    \end{equation}
\end{itemize}

For ease of notation, we will also use $\iota\hoc{\alpha}\colon p\hoc{\alpha} \to p \hoc{\beta}$ to represent the composite of inclusions for any pair of ordinals $\alpha < \beta$.   
We will also define $\lambda\hoc{\alpha}\colon p \tri p\hoc{\alpha} \to p \hoc{\alpha+1}$ to be the coproduct inclusion.

\begin{definition}
    For an ordinal $\kappa$, a polynomial $p$ is called \textbf{$\kappa$-small} if for all $\kappa$-filtered categories $\ja$ and all diagrams $Q\colon \ja \to \smset$, the natural map 
    \[
        \colim_{j \colon \ja} \left(p \tri Q_j\right) \to p \tri \left( \colim_{j \colon \ja} Q_j \right)
    \]
    is an isomorphism.
\end{definition}

\begin{remark}
A polynomial $p: \poly$ is $\kappa$-small if and only if all of its direction-sets have cardinality less than $\kappa$. It is called \textbf{finitary} if and only if it is $\omega$-small. For every polynomial $p$, there is some $\kappa$ for which $p$ is $\kappa$-small.
\end{remark}

\begin{definition}\label{def.free_monad}
    If $p$ is $\kappa$-small, then define $\free_p \coloneqq p\hoc{\kappa}$. 
\end{definition}

We will show in~\cref{thm:free-adjunction} that $\free_p$ is the free monad on $p$. Therefore, $\free_p$ is unique and hence well-defined.

Think of the positions of $\free_p$ as decision trees whose building blocks are the positions of $p$. For instance, suppose that a position $P: p(1)$ is a question and the directions $p[P]$ are possible answers to question $P$. Then $\yon$  is a unique question with a unique answer, which we think of as ``no further questions''. Then, consider the positions of $\free_p$ that factor through  $p\hoc{\alpha}$. For example, the positions of $p \hoc{2} = \yon + p \tri(\yon + p)$ is either (1) no further questions or (2) a question and for each possible answer either no further questions or another question. Therefore $p\hoc{2}$ represents interviews with \textit{at most two questions}, and in general for finite $i$, $p\hoc{i}$ represents interviews with at most $i$ questions.  The directions of $\free_p$ represent the possible paths through the decision tree. 

\begin{example}
    For finitary $p$, the positions of $\free_p$ are $p$-trees with finite height and whose directions are the dangling leaves. 
\end{example}

\begin{example}
    A finite $\yon$-tree is determined by its height and has one dangling leaf; thus, $\free_\yon \cong  \nn \yon$.
\end{example}

\subsection{Monad Structure on $\free_p$}

Let $\Mod(\poly)$ denote the category of monoids in $(\poly, \yon, \tri)$.  Note that when viewed as endofunctors on $\smset$, a $\tri$-monoid is in fact a monad on $\smset$. Therefore we refer to the objects of $\Mod(\poly)$ as $\tri$-\textbf{monoids}, \textbf{polynomial monads}, or simply \textbf{monads}.

Next we will give a $\tri$-monoid structure on $\free_p$.
The unit $\freeu{p} \colon \yon \to \free_p$ is defined to be the inclusion $\iota\hoc{0}\colon \yon = p\hoc{0} \to \free_p$. The multiplication is more complicated so we begin with the following lemma.

\begin{lemma}\label{lem:def-mu}
    For ordinals $\alpha, \beta$ there exist maps $\mu\hoc{\alpha, \beta} \colon p\hoc{\alpha} \tri p\hoc{\beta} \to p\hoc{\alpha + \beta}$ such that for all $\alpha' < \alpha$ and $\beta' < \beta$ the following diagram commutes
\begin{equation}\label{eq:p-alpha-beta}
\begin{tikzcd}[column sep=large]
	{p\hoc{\alpha'} \tri p \hoc{\beta'}} & {p\hoc{\alpha' + \beta'}} \\
	{p\hoc{\alpha} \tri p\hoc{\beta}} & {p\hoc{\alpha + \beta}}
	\arrow["{\mu\hoc{\alpha, \beta}}"', from=2-1, to=2-2]
	\arrow["{\mu\hoc{\alpha', \beta'}}", from=1-1, to=1-2]
	\arrow["{\iota\hoc{\alpha'} \tri \iota\hoc{\beta'}}"', from=1-1, to=2-1]
	\arrow["{\iota\hoc{\alpha' + \beta'}}", from=1-2, to=2-2]
\end{tikzcd}
\end{equation}
\end{lemma}
\begin{proof}
    We define the maps $\mu\hoc{\alpha, \beta}$ transinductively on $\alpha$. 
    
    For $\alpha = 0$, $\mu\hoc{\alpha, \beta}$ is the identity on $p\hoc{\beta}$.

    For successor ordinals $\alpha + 1$, suppose we have already defined $\mu\hoc{\alpha, \beta}$. Note that 
    \[
        p\hoc{\alpha +1} \tri p\hoc{\beta} = (\yon + p \tri p \hoc{\alpha}) \tri p\hoc{\beta} = p \hoc{\beta} + p \tri p\hoc{\alpha} \tri p\hoc{\beta}.     
    \] So we define $\mu\hoc{\alpha + 1,\beta}$ to be the copairing of the inclusion $\iota\hoc{\beta}\colon p\hoc{\beta} \to p \hoc{\alpha + \beta}$ with the map 
    \[
        p \tri p\hoc{\alpha}\tri p \hoc{\beta} \xrightarrow{p \tri \mu\hoc{\alpha, \beta}} p \tri p\hoc{\alpha + \beta} \xrightarrow{\lambda\hoc{\alpha + \beta}} p\hoc{\alpha + \beta + 1}.
    \]
    That the diagram in~\eqref{eq:p-alpha-beta} commutes for $\alpha' = \alpha +1$ can be shown inductively on $\alpha$. 

    Next suppose that $\alpha$ is a limit ordinal and suppose we have defined $\mu\hoc{\alpha', \beta} \colon p\hoc{\alpha'} \tri p\hoc{\beta} \to p\hoc{\alpha' + \beta}$ for $\alpha' < \alpha$. Then we define $\mu\hoc{\alpha, \beta}$ to be the composite
    \[
        p\hoc{\alpha} \tri p \hoc{\beta} = \left(\colim_{\alpha' < \alpha} p\hoc{\alpha'}\right) \tri p\hoc{\beta} 
        \iso \colim_{\alpha' < \alpha} (p\hoc{\alpha'} \tri p \hoc{\beta}) 
        \to \colim_{\alpha' < \alpha} p\hoc{\alpha' + \beta}
        \to p \hoc{\alpha + \beta}.
    \]
    The second isomorphism follows from~\cref{prop:colim-left-distribute}. That these maps make the diagram in~\eqref{eq:p-alpha-beta} commute is immediate.
\end{proof}

This lemma implies that for ordinals $\alpha, \beta < \kappa$ the maps
\[
    p\hoc{\alpha}\tri p \hoc{\beta} \xrightarrow{\mu\hoc{\alpha, \beta}} p \hoc{\alpha + \beta} \to \free_{p}
\] form a cocone into $\free_{p}$. Define the multiplication $\freem{p} \colon \free_p \tri \free_p \to \free_p$ using the universal property: 
\begin{align*}
    \free_p \tri \free_p &= \left(\colim_{\alpha < \kappa} p\hoc{\alpha} \right) \tri \left(\colim_{\beta < \kappa} p\hoc{\beta} \right)  \cong \colim_{\alpha < \kappa} \colim_{\beta < \kappa} (p \hoc{\alpha} \tri p\hoc{\beta}) \to \free_p.
\end{align*}
The second isomorphism follows from~\cref{prop:colim-right-distribute} and~\cref{prop:colim-left-distribute}. 
Explicitly, given a decision tree and for each dangling leaf another decision tree, the multiplication $\free_p \tri \free_p \to \free_p$ glues these together into a single decision tree. 
From this definition and Lemma~\eqref{lem:def-mu}, the following lemma is immediate.

\begin{lemma}\label{lem:mu-alpha-beta}
    Let $p$ be $\kappa$-small. Then for all $\alpha, \beta < \kappa$, the following diagram commutes.

\[\begin{tikzcd}
	{p\hoc{\alpha} \tri p \hoc{\beta}} & {p\hoc{\alpha + \beta}} \\
	{\free_p \tri \free_p} & {\free_p}
	\arrow["{\iota\hoc{\alpha + \beta}}", from=1-2, to=2-2]
	\arrow["\freem{p}"', from=2-1, to=2-2]
	\arrow["{\mu\hoc{\alpha, \beta}}", from=1-1, to=1-2]
	\arrow["{\iota\hoc{\alpha} \tri \iota \hoc{\beta}}"', from=1-1, to=2-1]
\end{tikzcd}\]
\end{lemma}

As a predecessor to proving the associativity law for $(\free_p, \freeu{p}, \freem{p})$ as a $\tri$-monoid, we present the following variant of~\cref{lem:def-mu} whose proof appears in~\cref{apx:free-monad}.

\begin{restatable}{lemma}{associativity}
\label{lem:associativity}
    For ordinals $\alpha, \beta, \gamma$, the following diagram commutes:
    \[\begin{tikzcd}[column sep=huge,ampersand replacement=\&]
    	{p\hoc{\alpha} \tri p\hoc{\beta} \tri p \hoc{\gamma}} \& {p\hoc{\alpha} \tri p \hoc{\beta + \gamma}} \\
    	{p\hoc{\alpha + \beta} \tri p \hoc{\gamma}} \& {p\hoc{\alpha + \beta + \gamma}}
    	\arrow["{\mu\hoc{\alpha, \beta} \tri p \hoc{\gamma}}"', from=1-1, to=2-1]
    	\arrow["{p\hoc{\alpha} \tri \mu\hoc{\beta,\gamma}}", from=1-1, to=1-2]
    	\arrow["{\mu\hoc{\alpha, \beta + \gamma}}", from=1-2, to=2-2]
    	\arrow["{\mu\hoc{\alpha + \beta, \gamma}}"', from=2-1, to=2-2]
    \end{tikzcd}\]
\end{restatable}

\begin{proposition}\label{prop:m_p}
    For every polynomial $p: \poly$, there is a $\tri$-monoid structure on $\free_p$, for which the unit and multiplication
    $
        \freeu{p} \colon \yon\to\free_p
    $
    and
    $
        \freem{p} \colon \free_p\tri\free_p\to\free_p
    $
    are defined as above. 
\end{proposition}
\begin{proof}
    First we show the left unit law. Since $\yon \tri \free_p$ is isomorphic to $\colim_{\beta < \kappa} (\yon \tri p\hoc{\beta})$, the left unit law follows from~\cref{lem:mu-alpha-beta} with $\alpha = 0$. Likewise, for the right unit law.

    Second we show associativity. Due to the isomorphism $$\free_p \tri \free_p \tri \free_p\cong\colim_{\alpha, \beta, \gamma < \kappa} \left(p\hoc{\alpha} \tri p \hoc{\beta} \tri p \hoc{\gamma}\right), $$ it suffices to show that for  $\alpha, \beta, \gamma < \kappa$, the outer diagram in the following commutes. 

\[\begin{tikzcd}[column sep=huge]
	{p\hoc{\alpha} \tri p\hoc{\beta} \tri p \hoc{\gamma}} & {p\hoc{\alpha} \tri p \hoc{\beta + \gamma}} & {\free_ p \tri \free_p} \\
	{p\hoc{\alpha + \beta} \tri p \hoc{\gamma}} & {p\hoc{\alpha + \beta + \gamma}} \\
	{\free_p \tri \free_p} && {\free_p}
	\arrow["{\mu\hoc{\alpha, \beta} \tri p \hoc{\gamma}}"', from=1-1, to=2-1]
	\arrow["{p\hoc{\alpha} \tri \mu\hoc{\beta,\gamma}}", from=1-1, to=1-2]
	\arrow["{\mu\hoc{\alpha, \beta + \gamma}}", from=1-2, to=2-2]
	\arrow["{\mu\hoc{\alpha + \beta, \gamma}}"', from=2-1, to=2-2]
	\arrow["{\iota\hoc{\alpha + \beta + \gamma}}"', from=2-2, to=3-3]
	\arrow["{\freem{p} }", from=1-3, to=3-3]
	\arrow["\freem{p}"', from=3-1, to=3-3]
	\arrow["{\iota\hoc{\alpha + \beta} \tri \iota \hoc{\gamma}}"', from=2-1, to=3-1]
	\arrow["{\iota\hoc{\alpha} \tri \iota\hoc{\beta + \gamma}}", from=1-2, to=1-3]
\end{tikzcd}\]
The upper left-hand square commutes by ~\cref{lem:associativity} and the other squares commute by ~\cref{lem:mu-alpha-beta}.
\end{proof}

\subsection{The Monad $\free_p$ is Free}

Now that we have verified that $\free_p$ is indeed a monad, we want to justify calling it the \textbf{free monad} by giving a left adjoint  $\free_-\colon \poly \to \Mod(\poly)$  to the forgetful functor $U\colon \Mod(\poly) \to \poly$ which takes a $\tri$-monoid $(q, \eta_q, \mu_q)$ to its carrier $q$.

We begin by defining the action of $\free_-$ on morphisms.
Let $f\colon p \to q$ in $\poly$. We will define $\free_f \colon \free_p \to \free_q$  by inductively defining morphisms $f\hoc{\alpha} \colon p\hoc{\alpha} \to q \hoc{\alpha}$ such that for all $\alpha< \beta$ the following diagram commutes:
\begin{equation}\label{eq:f-iota-commute}
\begin{tikzcd}
	{p\hoc{\alpha}} & {p\hoc{\beta}} \\
	{q\hoc{\alpha}} & {q\hoc{\beta}}
	\arrow["{f\hoc{\alpha}}"', from=1-1, to=2-1]
	\arrow["{f\hoc{\beta}}", from=1-2, to=2-2]
	\arrow["{\iota\hoc{\alpha}}", from=1-1, to=1-2]
	\arrow["{\iota\hoc{\alpha}}"', from=2-1, to=2-2]
\end{tikzcd}
\end{equation}

Define $f\hoc{0}\colon p\hoc{0} \to q\hoc{0}$ to be the identity on $\yon$. For a successor ordinal $\alpha + 1$, suppose that we have already defined $f\hoc{\alpha} \colon p \hoc{\alpha} \to q \hoc{\alpha}$. Then we define $f\hoc{\alpha + 1} \colon p \hoc{\alpha + 1} \to q \hoc{\alpha + 1}$ to be 
\[
    \yon + p \tri p\hoc{\alpha} \xrightarrow{ \yon + f \tri f\hoc{\alpha}} \yon + q \tri q \hoc{\alpha}.
\]
To show that the diagram in Equation~\eqref{eq:f-iota-commute} commutes for successor ordinals, it suffices to show that for all $\alpha$, the diagram with $\beta = \alpha + 1$ commutes. By induction this is immediate by the definitions of $\iota\hoc{\alpha + 1}$ and $f\hoc{\alpha + 1}$.

Suppose that $\alpha$ is a limit ordinal and that we have defined $f\hoc{\alpha'} \colon p \hoc{\alpha'} \to q \hoc{\alpha'}$ for all $\alpha' < \alpha$. For each $\alpha' < \alpha$ we have maps 
\[
    p\hoc{\alpha'} \xrightarrow{f\hoc{\alpha'}} q\hoc{\alpha'} \xrightarrow{\iota\hoc{ \alpha'}} q\hoc{\alpha}.
\] Since the diagram in Equation~\eqref{eq:f-iota-commute} commutes for all pairs of ordinals less than $\alpha$, these maps form a cocone into $q\hoc{\alpha}$. So, by the universal property of the colimit, $p\hoc{\alpha} = \colim_{\alpha' < \alpha} p \hoc{\alpha'}$, there is an induced map $f\hoc{\alpha}: p \hoc{\alpha} \to q\hoc{\alpha}$. That the diagram in Equation~\eqref{eq:f-iota-commute} commutes when $\beta$ is a limit ordinal follows from the uniqueness of the universal map. 

Let $\kappa$ be such that both $p$ and $q$ are $\kappa$-small. Then we define $\free_{f} \coloneqq f\hoc{\kappa}$. We defer the proof that $\free_{f}$ is a map of $\tri$-monoids to \cref{apx:free-monad}.

\begin{restatable}{proposition}{freef}
\label{prop:m_f}
    For $f \colon p \to q$ in $\poly$, the polynomial map $\free_f: \free_p \to \free_q$ is a map of $\tri$-monoids.
\end{restatable}
    
It is straightforward to show by transfinite induction that the action of $\free_-$ on morphisms is functorial.
Therefore,  \cref{prop:m_p} and \cref{prop:m_f} define a functor $\free_-: \poly \to \Mod(\poly)$. 

\begin{restatable}{theorem}{freeadjunction}
\label{thm:free-adjunction}
    There is an adjunction
    \[
        \adj{\poly}{\free_-}{U}{\Mod(\poly)}.
    \]
\end{restatable}
Therefore $\free_p$ is the \textbf{free monad} on the polynomial $p$ and $\free_-$ is the \textbf{free monad monad}. We defer the proof of \cref{thm:free-adjunction} and its requisite lemmas to \cref{apx:adjunction}.

Essentially we have shown that $p$ is $\kappa$-small implies that $p\hoc{\kappa}$ is \emph{the} free monad on $p$. Hence we have the following Corollary.

\begin{corollary}
    If $\kappa < \kappa'$ and $p$ is $\kappa$-small, then $p\hoc{\kappa}$  is isomorphic to $p \hoc{\kappa'}$.
\end{corollary}

Alternatively, we can show directly via transfinite induction that under these hypotheses, the inclusion $ \iota\hoc{\kappa'} \colon p\hoc{\kappa} \to p \hoc{\kappa'}$ is an isomorphism. If $\kappa$ is a limit ordinal and $\kappa' = \kappa + 1$,  then its inverse is the map
\[
    p\hoc{\kappa + 1} = \yon + p \tri p\hoc{\kappa} = \yon + p \tri \colim_{\alpha < \kappa} p\hoc{\alpha}  \xrightarrow{\iso} \colim_{\alpha < \kappa} \left(\yon + p \tri p \hoc{\alpha}\right) = \colim_{\alpha < \kappa} p \hoc{\alpha + 1} = p \hoc{\kappa}.
\]
The isomorphism follows from Proposition~\ref{prop:colim-right-distribute}. The remainder of the induction  is straightforward. 

This corollary implies that the natural map 
\[\free_p \To{\iso} \yon + p \tri \free_p.\] 
is an isomorphism. In terms of other words, a $p$-shaped interview is equivalent to either no interview or a question in $p(1)$ and for every answer, another $p$-shaped interview.

\section{Interactions Between Free Monad and Cofree comonad}\label{sec:module}
\subsection{The Cofree Comonad Comonad}

It is a beautiful fact that in $\poly$ the $\tri$-comonoids are categories and $\tri$-comonoid maps are cofunctors. Thus, we use $\catsharp$ to denote the category of $\tri$-comonoids and their maps. Dual to the construction of the free monad monad in Section~\ref{sec:free-monad}, here we define the cofree comonad and show that it is right adjoint to the forgetful functor $U\colon\catsharp\to\poly$ given by $U(c,\epsilon,\delta)\coloneqq c$. In \cref{app.cofree} we prove the statements presented in this Section.

\begin{restatable}{proposition}{propcofree}
\label{prop:cofree}
    There is a functor $\cofree_-\colon\poly\to\poly$ such that $\cofree_p$ has the structure of a $\tri$-comonoid for each $p:\poly$,
\[
    \cofree_p\to\yon
    \qqand
    \cofree_p\to\cofree_p\tri\cofree_p.
    \qedhere
\]
\end{restatable}

For a polynomial functor $q: \poly$, the positions of the cofree comonad are \textbf{$q$-behavior trees}. These are defined coinductively as a position  $Q: q(1)$ and a map from directions of $q[Q]$ to $q$-behavior trees.

%
\begin{restatable}{theorem}{cofreeadjunction}
\label{thm:cofree_comonad_comonad}
    There is an adjunction
\[
\adj{\catsharp}{U}{\cofree_-}{\poly}.
\]
\end{restatable}

Hence $\cofree_q$ is the \textbf{cofree comonad} on the polynomial $q$ and $\cofree_-$ is the \textbf{cofree comonad comonad}. 

\subsection{The Module Structure $\free_p \otimes \cofree_q \to \free_{p \otimes q}$}\label{sec.module}

For polynomials $p$ and $q$ consider the map 
\[
    p \otimes \cofree_q \to p \otimes q \to \free_{p \otimes q}
\]
where the first map is induced by the counit of $\cofree_q$ and the second map is the unit of $\free_{p\otimes q}$. This composite induces a map of polynomials $p \to [\cofree_q, \free_{p \otimes q}]$. By duoidality of $\otimes$ and $\tri$, the internal hom $[\cofree_q, \free_{p \otimes q}]$ is a $\tri$-monoid as well. Therefore, by the adjunction in \cref{thm:free-adjunction} the map of polynomials $p \to [\cofree_q, \free_{p \otimes q}]$ induces a map of $\tri$-monoids  $\free_p \to [\cofree_q, \free_{p \otimes q}]$, which is equivalent to a polynomial map
\[
    \modstruct_{p,q}\colon \free_p \otimes \cofree_q \to \free_{p \otimes q}.
\]
This map is the \textbf{free monad-comonad interaction law} described in~\cite[Section 3.3]{Katsumata2020interaction}. In \cref{apx:module} we prove the statements presented in this Section.

\begin{restatable}{proposition}{naturalinteraction}
    The maps $\modstruct_{p,q}$ are natural in $p$ and $q$.
\end{restatable}

Note that $\cofree$ is lax monoidal (see \cref{prop:cofree-monoidal})
\[
\cofree_p\otimes\cofree_q\to\cofree_{p\otimes q}
\qqand
\yon\to\cofree_{\yon}
\]
so in this sense we would say "matter can also take the place of pattern."

Recall from \cite{nlab:module_over_a_monoidal_functor} the notion of a module over a monoidal functor.

\begin{restatable}{theorem}{module}\label{thm.main}
    There is a left-module over $\cofree_-: (\poly, \otimes, \yon) \to (\poly, \otimes, \yon)$ consisting of:
    \begin{itemize}
        \item $\poly$ as a left module category over $(\poly, \otimes, \yon)$.
        \item The functor $\free_- \colon \poly \to \poly$.
        \item The natural transformation $\modstruct\colon \free_- \otimes \cofree_- \Rightarrow \free_{- \otimes -}$.
    \end{itemize}
\end{restatable}

\section{Applications}\label{sec:apps}
In this section, we will give four applications of the module structure introduced in Section~\ref{sec:module}. Each consists of:
\begin{itemize}
    \item A pattern of type $p$, represented by a map into $\free_p$.
    \item Matter of type $q$, represented by a map into $\cofree_q$.
    \item A \textit{runs on} map $p \otimes q \to r$.
\end{itemize}
Then the composite $\free_p \otimes \cofree_q \to \free_{p \otimes q} \to \free_r$ represents the interaction \textit{pattern runs on matter}.

\subsection{Interviews Run on People}

\newcommand{\tea}{\mathrm{Tea?}}
\newcommand{\kind}{\mathrm{Kind?}}
\newcommand{\yes}{\mathrm{yes}}
\newcommand{\no}{\mathrm{no}}
\newcommand{\green}{\mathrm{green}}
\newcommand{\black}{\mathrm{black}}
\newcommand{\herbal}{\mathrm{herbal}}

As in Section~\ref{sec:free-monad}, suppose that the positions of $p$ are questions and the directions are possible answers to each question. For example consider the polynomial,
\(
    p = \{\tea\} \yon ^{\{\yes, \no\}} + \{\kind\} \yon ^{\{\green, \black, \herbal\}}
\) which we can view as the corolla forest: 
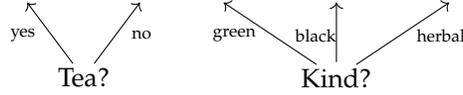
\[
\begin{tikzpicture}[trees, 
  level distance=10mm,
  level 1/.style={sibling distance=15mm},
  level 2/.style={sibling distance=2.5mm},
  level 3/.style={sibling distance=1.25mm}]
	\node (a) {\small$\tea$}
    	child {edge from parent node[left] {\tiny$\yes\quad $}}
    	child {edge from parent node[right] {\tiny$\quad\no$}}
		;
        \node[right = 1in of a] (b) {\small $\kind$}
		  child {edge from parent node[left] {\tiny $\green\quad $}}
		  child {edge from parent node[left] {\tiny $\black $}}
            child {edge from parent node[right] {\tiny$\quad\herbal $}}
		;
\end{tikzpicture}
\]
Essentially, $p$ consists of the questions:
\begin{itemize}
    \item ``Do you want tea?'', with possible answers ``yes'' or ``no''.
    \item ``What kind of tea do you like?'', with possible answers ``green'', ``black'', or ``herbal''.
\end{itemize}
Consider the pattern $\yon \to \free_p$ which selects the following interview:
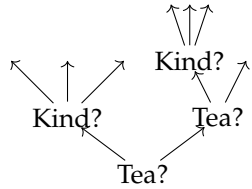
\[
\begin{tikzpicture}[trees, 
  level 1/.style={sibling distance=20mm},
  level 2/.style={sibling distance=7.5mm},
  level 3/.style={sibling distance=2.5mm}]
	\node {\small $\tea$}
    	child {node {\small $\kind$}
                child
                child
                child
            } 
    	child {node {\small $\tea$}
                child {node {\small $\kind$}
                    child
                    child
                    child
                }
                child 
            }
		;
\end{tikzpicture}
\]

The polynomial $[p,\yon] = \left(\prod_{P:p(1)} p[P]\right) \yon ^{p(1)}$ is the \textbf{universal answerer} for the polynomial $p$. Its positions are a choice of answer for each question in $p(1)$ and its directions are the questions $p(1)$. A person is a map $\yon \to \cofree_{[p, \yon]}$. In other words, a person is a behavior tree in which a node is an answer for each question and a branch is one such question. Each answer may depend on the sequence of questions asked so far.  Such a map can always be factored as $\yon \to S\yon ^S \to \cofree_{[p, \yon]}$. In other words, a person consists of (1) a dynamical system which outputs answers to $p$ questions and inputs $p$ questions and (2) an initial condition.

Note  \cite[Section 2.5]{Katsumata2020interaction} refers to $[p, \yon]$ as the dual of $p$.  We will see in \cref{subsec.games} that it is also interesting to consider generalizations of this duality to $[p,t]$ for any polynomial monad $t$. Indeed, for any monad $t$, there is a map
\begin{equation}\label{eqn.generalize_dual}
	\free_p\otimes\cofree_{[p,t]}\to\free_t\to t
\end{equation}

The evaluation map $p \otimes [p, \yon] \to \yon$ describes how to run a question on a universal answerer. Given an interview and a person we get the composite
\[
    \yon \to \free_p \otimes \cofree_{[p, \yon]} \to \free_{p \otimes [p, \yon]} \to \free_\yon = \nn \yon, 
\]
which sends the single position of $\yon$ to the number of questions asked when the interview is run on the person. For example, consider Alice who always responds that she does not want tea and she likes herbal tea. Running the tea interview on Alice results in two questions: ``$\tea$'' then ``$\tea$''. Conversely, consider Bob who  at first politely declines  tea and likes black tea. Then after the first question, he responds that he actually does want tea and he likes black tea. Running the tea interview on Bob results in three questions: ``$\tea$'', ``$\tea$'', then ``$\kind$''. 


\subsection{Programs Run on Operating Systems}\label{sec:programs}

Consider the following program.
\begin{minted}{python}
def guessing_game(max_guesses, goal):
    if max_guesses==0: 
        return False
    guess=read()
    if guess==goal:
        return True
    return guessing_game(max_guesses-1, goal)
\end{minted}

We represent the argument/return type with the polynomial $r = \sum_{m: \nn, g: \nn} \yon^{\Bool}$ where $m$ represents the variable \mintinline{python}{max_guesses} and $g$ represents the variable \mintinline{python}{goal}. We represent the effect type that reads in natural numbers with the polynomial $p = \yon^\nn$. 
Finally, we will represent this program with a map $r\to \free_p$ that we  define inductively using the decomposition $r = \sum_{m: \nn} r_m$ with $r_m = \sum_{g: \nn} \yon^\Bool$. We start by defining the following maps:
\begin{itemize}
    \item Define a map $r_0 \to \yon$. Each position of $r_0$ is sent to the single position of $\yon$. The single direction of $\yon$ is sent to $\texttt{False}: \Bool$. 
    \item Define a map $r_{m+1} \to p \tri (\yon + r_{m})$.  On positions, for each position $g: r_{m +1}(1)$, we need a  function from $\nn$ to the positions of $\yon + r_{m}$. In particular, we send $g$ to the map sending $g' \colon \nn$ to $\yon$ if $g = g'$ and to the position $g: r_{m}(1)$ otherwise. Then on directions, if $g = g'$, the single direction of $\yon$ is sent to $\texttt{True}: \Bool$ and if $g \neq g'$, then we use the identity on directions. 
\end{itemize}
Now we will define maps $r_m \to \free_p$ inductively using the isomorphism $\free_p \iso \yon + p \tri \free_p$ and the inclusion $\iota\hoc{0}\colon \yon \to \free_p$. As a base case, we have the composite
\(
    r_0 \to \yon \xrightarrow{\iota\hoc{0}} \free_p.
\)Given $r_m \to \free_p$, we define
\[
    r_{m+1} \to p \tri (\yon + \free_p) \xrightarrow{p\tri (\iota\hoc{0}, \free_p)} p \tri \free_p \to \free_p.
\]
A operating system with effects in $[p, \yon] \cong \nn \yon$ is a map $\yon \to \cofree_{[p, \yon]}\cong(\nn\yon)^\nn$. It consists of a stream of natural numbers, which are the responses it will give to the \mintinline{python}{read()} effect.

Using the interaction $\modstruct_{p, [p, \yon]}$ and the evaluation map $p \otimes [p, \yon] \to \yon$, we get the composite
\[
    r \cong r \otimes \yon \to \free_p \otimes \cofree_{[p,\yon]} \to \free_{p \otimes [p, \yon]} \to \free_{\yon} \cong \nn \yon,
\]
which expresses how the program runs on the chosen operating system. On positions it maps $(m:\nn, g:\nn)$ to the minimum of $m$ and the number of responses  the operating system takes to guess the goal $g$. On directions, it maps the single direction to \texttt{True} if the goal was guessed in at most $m$ guesses and to \texttt{False} otherwise.

\subsection{Voting Schemes Run on Voters}

For a finite set of candidates $X$, consider the polynomial $p = \sum_{A \subseteq X}\yon^A$. A position of $p$ is a ballot consisting of some subset of the candidates.  For such a position, its  directions are possible winners of the ballot. A voting scheme with $M$ voters is a map $p \to \free_{\bigotimes_M p}$ where $\bigotimes_0p\coloneqq \yon$ and $\bigotimes_{M+1}p\coloneqq p\otimes\bigotimes_Mp$. On positions, a subset of candidates $A \subseteq X$ maps to a terminating decision tree in which each node is a personalized ballot given to each of the $M$ voters and each branch corresponds to the tuple  of each voter's selection. On directions, each leaf of this decision tree  maps to an overall winner. 

\textbf{Exhaustive run-off} is a voting scheme in which each voter selects their preference from the $A$ candidates, and then the candidates with the fewest number of votes are eliminated. If only a single candidate remains, then they are elected the winner. Otherwise, another round of voting proceeds with the remaining candidates. This voting scheme can be encoded into a polynomial map $p \to \free_{\bigotimes_M p}$. As in Section~\ref{sec:programs} we define the map inductively. Consider the decomposition $p = \sum_{n = 0}^{\#X} p_n$ where $p_n = \sum_{A \subseteq X, \#A = n} \yon^A$. Then consider the following maps:
\begin{itemize}
    \item Note that $p_0 = 1$ is isomorphic to $\bigotimes_M p_0 = \bigotimes_M 1$. There is an inclusion of $p_0$ into $p$ and hence a map $p_0 \iso \bigotimes_M p_0 \to \bigotimes_M p$.
    \item There is a unique map $p_1 \to \yon$.
    \item Define $p_{n + 1} \to \left(\bigotimes_{M} p\right) \tri \left(\sum_{k = 0}^{n} p_k\right)$ as follows. On positions $A \subseteq X$ maps to the product $\prod_M A: \prod_{M} p(1)$ and the map $\prod_{M} A \to \sum_{k = 0}^n p_k(1)$ defined by \[
    (a_1,\cdots, a_M) \mapsto A' \coloneqq A \setminus \argmin \#_{(a_1, \cdots, a_n)}\] where $\#_{(a_1, \cdots, a_n)}\colon A \to \nn$ counts the number of votes for each candidate and so $\argmin \#_{(a_1, \cdots, a_n)}$ is the \emph{set} of candidates with the fewest votes. On directions, it is the inclusion $A' \to A$.
\end{itemize}
Now, we define $p \to \free_{\bigotimes_M p}$ inductively as follows. As base cases we have the composites
\[
    p_0 = 1 = \bigotimes_M 1 \to \bigotimes_M p \to \free_{\bigotimes_M p}
\qqand
    p_1 \to \yon \to \free_{\bigotimes_M p}.
\]
Given a maps $p_k \to \free_{\bigotimes_M p}$ for $k < n+1$, we define
\[
    p_{n +1} \to \left(\bigotimes_{M} p\right) \tri \left(\sum_{k = 0}^{n} p_k\right) \to \free_{\bigotimes_M p} \tri \left(\sum_{k = 0}^{n} \free_{\bigotimes_M p}\right) \to \free_{\bigotimes_M p} \tri \free_{\bigotimes_M p} \xrightarrow{\mu} \free_{\bigotimes_M p}.
\]

A  voter selects a candidate from every subset of candidates. Such a voter is represented by a polynomial map $\yon \to [p, \yon]$. Running $M$ elections on $M$ voters is represented by the composite 
\(
    \left(\bigotimes_M p\right) \otimes \left(\bigotimes_M [p, \yon] \right)  \iso  \bigotimes_M  (p \otimes [p, \yon]) \to \bigotimes_M \yon = \yon.
\) Therefore, given $M$ voters, we get a composite
\[
    p \to \free_{\bigotimes_M p} \otimes \left(\bigotimes_M [p, \yon]\right) \to \free_{\bigotimes_M p} \otimes \cofree_{\bigotimes_M [p, \yon]} \to \free_\yon = \nn \yon.
\]
On positions it maps a set of candidates $A \subseteq X$ to the number of run-offs required to elect a candidate. On directions, it maps the single direction of $\yon$ to the winner.

It is tempting to expect that the maps $p \to \free_{\bigotimes_M p}$ define an operad enriched in the Kleisli category $\poly_{\free}$ in the sense of~\cite{shapiro2022dynamic}. However exhaustive run-off is gerrymander-able meaning that the division of voters into districts can affect the end-result of the election. This observation suggests the following definition of gerrymandering.

\begin{definition}
    A voting scheme $p \to \free_{\bigotimes_M p}$ can be \textbf{gerrymandered} if and only if it does not extend to an operad enriched in $\poly_{\free}$.
\end{definition}

\subsection{Games Run on Players}\label{subsec.games}

\newcommand{\xs}{\times}
\newcommand{\os}{{\bigcirc}}
\newcommand{\ds}{-}
\newcommand{\ttt}{T}

For a game such as tic-tac-toe, let $p$ be the polynomial whose positions are game states and whose directions are next possible moves. Then we can represent the game play as a position in $\free_p$. In the game tic-tac-toe, a  game state is a placement of $\xs$'s and $\os$'s on a $3\times 3$ grid. In other words, it is a map $b: 9 \to \{\xs, \os, \ds\}$ where $\ds$ represents an open grid position.  For $m = 1,\ldots, 9$, let $B_m$ be the set of valid board states with $m$ open positions. Assuming that $\xs$ always plays first, these are
\[
  B_m \coloneqq \{b: 9 \to \{\xs, \os, \ds\} \mid \#(b\inv(\ds)) = m, \#(b\inv(\os)) \leq \#(b \inv(\xs)) \leq \#(b\inv(\os)) + 1\}.
\] If there are an odd number of open positions, then it is $\xs$'s turn. Otherwise, it is $\os$'s turn.

Given a board state $b: B_m$, the next possible moves are the open positions $b\inv(\ds)$. Therefore,
the polynomial $p_\xs = \sum_{b: B_1 + B_3 + \cdots + B_9} \yon ^{b\inv(\ds)}$ represents the board states and next possible moves for $\xs$. Likewise the polynomial  $p_\os = \sum_{b: B_2 + B_4 + \cdots + B_8} \yon^{b\inv(\ds)}$ represents the board states and next possible moves for $\os$.  

Consider the polynomial $p = p_\xs \tri (\yon + p_\os)$. This polynomial represents an $\xs$ move followed by either game over or an $\os$ move. There exists a polynomial map $\ttt \colon \yon^{\{\xs, \os, \ds\}} \to \free_p$ which selects the decision-tree in $\free_p$ corresponding to the rules of tic-tac-toe, and maps directions to the winner of a completed game or $\ds$ if the game is tied.
We will define the map $\ttt$ inductively. 
For $m = 1,\cdots, 9$, let $r_m = B_m \yon^{\{\xs, \os, \ds\}}$. 
\begin{itemize}
    \item Define $r_1 \to \yon$, as follows. Given a board $b: B_1$ with a single open position, send the single direction in $\yon$ to either the winner or to $\ds$ if the game is tied.
    \item Define $r_{m+1} \to p_\os \tri(\yon + r_m)$ for odd $m$, and define $r_{m+1} \to p_\xs \tri (\yon + r_m)$ for even $m$, as follows. On positions send a board state $b: B_{m+1}$ to itself in $p_\os(1)$ (resp. $p_\xs(1)$). Then for each valid move $m: b\inv(\ds)$ let $b': B_m$ be the updated board state. If $b'$ contains a winner, then send $m$ to $\yon$ and send the single direction of $\yon$ to the winner. Otherwise, send $m$ to $b': r_m(1)$ and let the map on directions be the identity.
\end{itemize}
Using these maps, we can inductively define maps $r_m \to \free_p$ for odd $m$ as follows. As a base case we have the composite $r_1 \to \yon \to \free_p$. Given a map $r_m \to \free_p$, we have the composite
\begin{align*}
    r_{m+2} \to  p_\xs \tri (\yon + p_\os \tri(\yon + r_m)) \to p_\xs \tri (\yon + p_\os \tri(\yon + \free_p)) &\to p_\xs \tri (\yon + p_\os \tri \free_p) \\ 
   & \to p_\xs \tri (\yon + p_\os) \tri \free_p = p \tri \free_p \to \free_p.
\end{align*}
The second to last map is induced by the composite
\[
    \yon + p_\os \tri \free_p = \yon \tri \yon + p_\os \tri \free_p \to \yon \tri \free_p + p_\os \tri \free_p \to (\yon + p_\os ) \tri \free_p.
\]
Below is the image of a board state in $r_{3}(1)$ under the map $r_3 \to \free_p$. From left to right the directions are sent to $\os$, $\ds$, $\ds$, $\ds$, and $\xs$, as these are the win/loss/tie results of the games as shown.
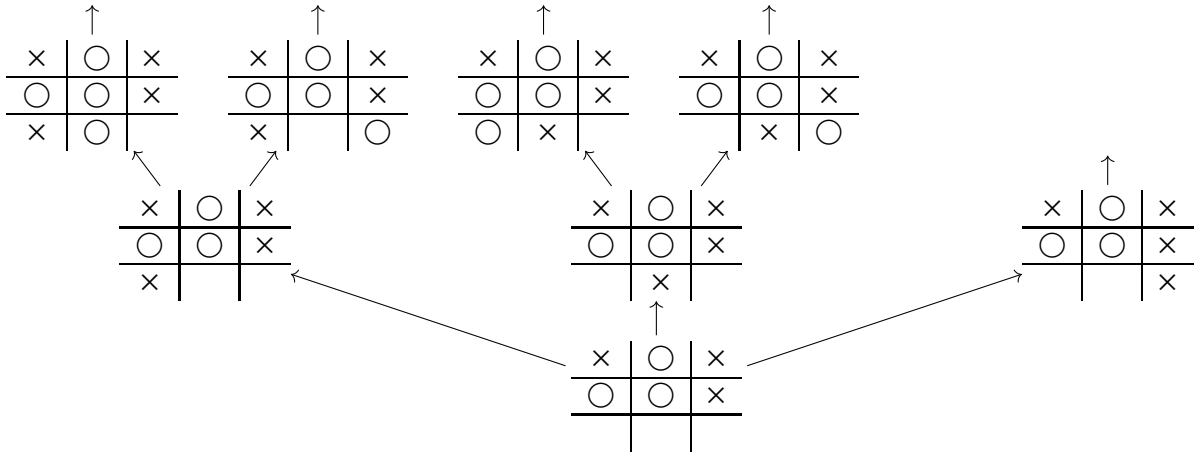
\[
\begin{tikzpicture}[trees, 
  level distance=20mm,
  level 1/.style={sibling distance=60mm},
  level 2/.style={sibling distance=30mm},
  level 3/.style={sibling distance=2.5mm}]
	\node (a) {\begin{tabular}{c|c|c}
                      $\xs$ & $\os$ & $\xs$ \\      \hline
                      $\os$ & $\os$ & $\xs$ \\      \hline
                            &       &  
                    \end{tabular}
                  }
    	child {node {\begin{tabular}{c|c|c}
                          $\xs$ & $\os$ & $\xs$ \\      \hline
                          $\os$ & $\os$ & $\xs$ \\      \hline
                          $\xs$ &       &  
                        \end{tabular}
                    }
                    child {node {\begin{tabular}{c|c|c}
                          $\xs$ & $\os$ & $\xs$ \\      \hline
                          $\os$ & $\os$ & $\xs$ \\      \hline
                          $\xs$ & $\os$ &  
                        \end{tabular}
                    } child[level distance=12mm] } 
                    child {node {\begin{tabular}{c|c|c}
                          $\xs$ & $\os$ & $\xs$ \\      \hline
                          $\os$ & $\os$ & $\xs$ \\      \hline
                          $\xs$ &       & $\os$
                        \end{tabular}
                    } child[level distance=12mm] }} 
    	child {node {\begin{tabular}{c|c|c}
                          $\xs$ & $\os$ & $\xs$ \\      \hline
                          $\os$ & $\os$ & $\xs$ \\      \hline
                                & $\xs$ &  
                        \end{tabular}
                    }
                    child {node {\begin{tabular}{c|c|c}
                          $\xs$ & $\os$ & $\xs$ \\      \hline
                          $\os$ & $\os$ & $\xs$ \\      \hline
                          $\os$ & $\xs$ &  
                        \end{tabular}
                    } child[level distance=12mm] } 
                    child {node {\begin{tabular}{c|c|c}
                          $\xs$ & $\os$ & $\xs$ \\      \hline
                          $\os$ & $\os$ & $\xs$ \\      \hline
                                & $\xs$ & $\os$
                        \end{tabular}
                    } child[level distance=12mm] } }
            child {node {\begin{tabular}{c|c|c}
                          $\xs$ & $\os$ & $\xs$ \\      \hline
                          $\os$ & $\os$ & $\xs$ \\      \hline
                                &       & $\xs$
                        \end{tabular}
                    }
                    child[level distance=12mm]} 
		;
\end{tikzpicture}
\]
Define $\ttt$ to be the map $r_9 = \yon ^{\{\xs, \os, \ds\}} \to \free_p$.

\newcommand{\lott}{\mathtt{lott}}

Let $(t, \eta, \mu)$ be a $\tri$-monad. A polynomial map $\varphi_\xs\colon\yon \to [p_\xs, t]$ represents an $\xs$ player and a polynomial map $\varphi_\os\colon\yon \to [p_\os, t]$ represents an $\os$ player. If $t$ is the trivial monad $\yon$, then for each board state, the player's next move is deterministic. If $t$ is the lottery monad $\lott = \sum_{M: \nn} \sum_{P: \Delta M} \yon ^M$, then the players' next moves are stochastic. 

By duoidality we have 
\[
    p \otimes [p_\xs, t] \otimes [p_\os, t] \to p \otimes [p_\xs, t] \otimes [\yon + p_\os, t] \to (p_x \otimes [p_x , t]) \tri ((\yon + p_\os) \otimes [\yon + p_\os, t]) \to t \tri t \to t
\] where the first map is induced by the monoidal unit of $t$. Given an $\xs$ player and an $\os$ player, we have a map $\yon^{\{\xs, \os, \ds\}}\cong 
    \yon^{\{\xs, \os, \ds\}}\otimes\yon\otimes\yon \To{T\otimes!\otimes!} 
    \free_{p} \otimes \cofree_\yon\otimes\cofree_\yon$.
Then, the game play is represented by the composite
\begin{equation}\label{eqn.game_play}
    \yon^{\{\xs, \os, \ds\}}\to 
    \free_{p} \otimes \cofree_\yon\otimes\cofree_\yon \To{\free_p\otimes\cofree_{\varphi_\xs}\otimes\cofree_{\varphi_\os}} 
    \free_{p} \otimes \cofree_{[p_\xs, t]} \otimes \cofree_{[p_\os, t]} \to 
    \free_{p \otimes [p_\xs, t] \otimes [p_\os, t]} \to 
    \free_t \to t
\end{equation}
which maps directions in $t$ to winners of completed games.

We can promote this setup to players which \emph{learn}, in other words  players whose strategy dynamically changes after each completed game. A dynamic $\xs$ player consists of a set of states $S_\xs$ and a polynomial map $S_\xs \yon^{S_\xs} \to \yon ^{\{\xs, \os, \ds\}} \otimes \cofree_{[p_\xs, t]} $. Likewise for a dynamic $\os$ player. Such a player consists of the following:
\begin{itemize}
    \item For each state, a behavior tree describing the player's strategy.
    \item For a winner or tied game and each finite path of the behavior tree, a new state.
\end{itemize}
Creating maps like $\varphi$ is the subject of reinforcement learning \cite{sutton2018reinforcement}. As a typical such algorithm, take $S_\xs$ to be the set of functions that assign a score in $\nn$ to each move (direction in $p_\xs$). Let $S_\xs\yon^{S_\xs} \to \yon ^{\{\xs, \os, \ds\}} \otimes \cofree_{[p_\xs, \lott]}$ be defined as follows:
\begin{itemize}
    \item On positions, it takes a score for each move and assigns a (stochastic) strategy which selects moves based on their relative scores.
    \item On directions, note that a direction of $\cofree_{[p_\xs, \lott]}$ contains a finite number of moves played by $\xs$ in the game. If the winner is $\xs$, then we add $1$ to the score for each played move. Otherwise, we keep the original scores.
\end{itemize}

Using the copy-on-directions map $\yon^{\{\xs, \os, \ds\}} \otimes \yon^{\{\xs, \os, \ds\}} \to \yon^{\{\xs, \os, \ds\}}$, we have the composite
\begin{multline*}
    S_\xs \yon^{S_\xs} \otimes S_\os \yon^{S_\os} \to  \yon ^{\{\xs, \os, \ds\}}\otimes \yon ^{\{\xs, \os, \ds\}}  \otimes \cofree_{[p_\xs, t]} \otimes\cofree_{[p_\os, t]} 
    \\
    \to \yon ^{\{\xs, \os, \ds\}}  \otimes \cofree_{[p_\xs, t]} \otimes\cofree_{[p_\os, t]} 
    \To{T} \free_p \otimes \cofree_{[p_\xs, t]} \otimes\cofree_{[p_\os, t]} 
    \to t.
\end{multline*}
This map takes a state for each player, runs the game using each player's strategy, and returns an updated state for each player based on the winner and the player's strategy.

\section{Conclusion}

In this paper, we constructed the free monad $\free_p$ and cofree comonad $\cofree_q$ on arbitrary polynomial functors $p,q:\poly$, and defined a module structure $\free_p\otimes\cofree_q\to\free_{p\otimes q}$. We also gave a series of examples to explain how this models the intuition ``pattern runs on matter.''

From here, it is not hard to show that $\free$ and $\cofree$ respectively extend to a monad and a comonad on $\oorg$, the double category which serves as the base of enrichment for dynamic categorical structures for deep learning, prediction markets, etc., as defined in \cite{shapiro2022dynamic}. We (or others) may show in future work that for any polynomial monad $t$, there is a functor $[-,t]\colon\oorg_\free\op\to\oorg^\cofree$ from opposite of the $\free$-Kleisli category to the $\cofree$-coKleisli category. The latter offers the ability for different machines to operate at different rates in wiring diagrams, and the former offers the ability to call multiple subprocesses before returning, though we have not found compelling examples of either; this again is future work.

The constructions in this paper should generalize straightforwardly to free monads and cofree comonads for familial functors between copresheaf categories, as in \cite{lynch2023concepts, weber2007familial}. One should check that there is again a module structure of the same form in that setting.

\bibliographystyle{eptcs}
\bibliography{references}

\appendix

\section{Distributing Colimits over $\tri$}

In general, colimits do not distribute over the subsitution product $\tri$. However, in this section we give hypotheses under which the natural maps
\[
    \colim_{j \in \ja} (p \tri q_j) \to p \tri (\colim_{j \in \ja} q_j)
    \qqand
    \colim_{j \in \ja} (p_j \tri q) \to (\colim_{j \in \ja} p_j) \tri q
\] are in fact isomorphisms.

\subsection{Right Distribution}
We begin with hypotheses under which colimits right distribute over $\tri$.

\begin{proposition}\label{prop:colim-right-distribute}
    If a polynomial $p$ is $\kappa$-small, then for all $\kappa$-filtered categories $\ja$ and diagrams $q\colon \ja \to \poly$, the natural map 
    \[
        \colim_{j \colon \ja} \left(p \tri q_j\right) \to p \tri \left( \colim_{j \colon \ja} q_j \right)
    \]
    is an isomorphism.
\end{proposition}
\begin{proof}
    That this map is an isomorphism on positions follows directly from the definition of $\kappa$-small. That this map is an isomorphism on directions follows from the definition of colimits in $\poly$ and the fact that connected limits preserve coproducts.
\end{proof}

\subsection{Left Distribution} \label{app:left-dist}

Next we consider hypotheses under which colimits left distribute over $\tri$.
Let $\polycart$ be subcategory of $\poly$ consisting of cartesian maps. If $\phi$ is cartesian, then so is $p\tri\phi$ and $p+\phi$ for any $p:\poly$.

First we give a series of Lemmas which allow us to prove Proposition~\ref{prop:colim-left-distribute}.

\begin{lemma}\label{lemma.ext_ff}
A polynomial map $p\to q$ is an isomorphism if and only if for all sets $X$, the induced function $p\tri X\to q\tri X$ is a bijection.
\end{lemma}
\begin{proof}
This is just a restatement of the fact that $\ext\colon\poly\to\smset^\smset$ is fully faithful; indeed, $\ext(p)(X)\cong p\tri X$.
\end{proof}

The coproduct completion $\Sigma\cat{C}$ of a category $\cat{C}$ has as objects pairs $(S,C)$ where $S:\smset$ and $C\colon S\to\cat{C}$ is a discrete diagram, and it has as morphisms lax triangles
\[
\begin{tikzcd}[column sep=10pt]
	S_1\ar[rr, "f"]\ar[dr, bend right=15pt, "C_1"', "" name=C1]&&
	S_2\ar[dl, bend left=15pt, "C_2", ""' name=C2]\\&
	\cat{C}
	\ar[from=C1,to=C2|-C1, shorten=5pt, Rightarrow, "f^\flat"]
\end{tikzcd}
\]
A diagram $\cat{I}\to\Sigma\cat{C}$ consists of a pair $(S,C)$, where $S\colon\cat{I}\to\smset$ is a diagram and $C\colon\el S\to\cat{C}$ is a functor  from the category of elements of $S$ to $\cat{C}$. 

\begin{lemma}\label{lemma.coprod_comp_colimits}
Let $\cat{C}$ be a category and let $\Sigma\cat{C}$ be its coproduct completion. If $\cat{C}$ has $\cat{I}$-shaped colimits, then so does $\Sigma\cat{C}$, and they are computed as follows.

Given a diagram $(S,C)\colon\cat{I}\to\Sigma\cat{C}$, consider the Kan extension
\[
\begin{tikzcd}
	\el S\ar[r, "C", ""' name=C]\ar[d]&
	\cat{C}\\
	\pi_0(\el S)\ar[ur, bend right=20pt, "\fun{Lan}"', "" name=colim]
	\ar[from=C, to=C|-colim, Rightarrow]
\end{tikzcd}
\]
where $\pi_0\colon\smcat\to\smset$ is the connected components reflection. Then the colimit is given by
\begin{equation}\label{eqn.colim_pi0_lan}
\colim(S,C)\cong(\pi_0(\el S),\fun{Lan}).
\end{equation}
For each element $s_0:\pi_0(\el S)$, its image object $\fun{Lan}(s_0):\ob\cat{C}$ is given by the colimit over all $(i,s_i)\mapsto s$ of $C(s_i)$.
\end{lemma}
\begin{proof}
First consider the case $\cat{C}=1$, so that $\Sigma\cat{C}\cong\smset$. We need to check that for any $S\colon\cat{I}\to\smset$, there is a bijection
\begin{equation}\label{eqn.colim_pi0}
\colim S\cong\pi_0(\el S).
\end{equation}
This follows from the fact that the inclusion of diagrams of sets into diagrams of categories is a left adjoint, the Grothendieck construction taking diagrams of categories to categories is a left adjoint, and $\pi_0$ taking categories to sets is a left adjoint.

Given $T:\smset$, a discrete diagram $D\colon T\to\cat{C}$, and diagrams
$(S_i,C_i)\to(T,D)$, coherently over $i:\cat{I}$, one obtains a diagram as left, which factors uniquely as right:
\[
\begin{tikzcd}[column sep=10pt]
	\el S\ar[rr]\ar[dr, bend right=15pt, "C"', "" name=C]&&
	T\ar[dl, bend left=15pt, "D", ""' name=D]\\&
	\cat{C}
	\ar[from=C, to=C-|D, shorten=5pt, Rightarrow]
\end{tikzcd}
\hspace{.6in}
\begin{tikzcd}
	\el S\ar[r]\ar[dr, bend right=15pt, "C"', "" name=C]&
	\pi_0(\el S)\ar[d, "\fun{Lan}" description, "" name=L, ""' name=L']\ar[r]&
	T\ar[dl, bend left=15pt, "D", ""' name=D]\\&
	\cat{C}
	\ar[from=C|-L', to=L', shorten=5pt, Rightarrow]
	\ar[from=L, to=L-|D, shorten=5pt, Rightarrow]
\end{tikzcd}
\]
This gives \eqref{eqn.colim_pi0_lan}. For the last statement, note that since $1\to\pi_0(\el S)$ is a map of sets, the pullback of categories
\[
\begin{tikzcd}
  \bullet\ar[r]\ar[d]\ar[dr, phantom, very near start, "\lrcorner"]&
  \el S\ar[d]\\
  1\ar[r, "s_0"']&
  \pi_0(\el S)
\end{tikzcd}
\]
is also a comma square. Hence it is exact in the sense of \cite{nlab:exact_square}, completing the proof.
\end{proof}

Lemma~\ref{lemma.coprod_comp_colimits} provides a formula for computing colimits in $\poly$, since $\poly\cong\Sigma\smset\op$ is the coproduct completion of $\smset\op$. Namely, the position-set of the colimit of $p\colon\cat{I}\to\poly$ is given by the colimit $\colim_ip_i(1)$ of the position-sets, and the directions at a position $P:\colim_ip_i(1)$ are given by the limit (taken in $\smset$, i.e.\ the colimit taken in $\smset\op$) of the associated connected diagram $\cat{E}_P$ of direction sets.
\[
\begin{tikzcd}
	\cat{E}_P\ar[r]\ar[d]&
	\el p(1)\ar[r, "{p[-]}", ""' name=C]\ar[d]&
	\smset\op\\
	1\ar[r, "P"']&
	\colim_ip_i(1)\ar[ur, bend right=20pt, "\fun{Lan}"', "" name=colim]
	\ar[from=C, to=C|-colim, Rightarrow]\ar[ul, phantom, very near end, "\lrcorner"]
\end{tikzcd}
\]

\begin{proposition}\label{prop.polycart_reflect}
The inclusion $\polycart\to\poly$ reflects all colimits.
\end{proposition}
\begin{proof}
Suppose given a category $I$, a functor $p\colon I\to\polycart$, a polynomial $\ol{p}:\poly$, and cartesian maps $\varphi_i\colon p_i\to \ol{p}$ forming a cone. Suppose further that in $\poly$ it is a colimit cone. To see that it is a colimit cone in $\polycart$ we need only check that for any $q$ and cone $\psi_i\colon p_i\to q$ in $\polycart$, the induced map $\psi\colon\ol{p}\to q$ in $\poly$ is in fact in $\polycart$, i.e.\ for any position $P:\ol{p}(1)$, the induced function $\psi^\sharp_P\colon q[\varphi(P)]\to\ol{p}[P]$ is bijective. 

By Lemma~\ref{lemma.coprod_comp_colimits}, the position-set of a colimit is the colimit of the position-sets for any diagram of polynomials, meaning there exists some $i:I$ and $P_i:p_i(1)$ representing $P$, i.e.\ with $\varphi_i(P_i)=P$. Since $p_i\to\ol{p}$ is cartesian, we have a commuting triangle of functions $q[\varphi(P)]\to\ol{p}[P]\to p_i[P_i]$ of which two are bijections; it follows that the required one is as well.
\end{proof}

\begin{proposition}\label{prop.polycart_colimits}
The inclusion $\polycart\to\poly$ creates coproducts and filtered colimits. Moreover, the composite
\[
\polycart\to\poly\To{\ext}\smset^\smset
\]
creates coproducts and filtered colimits.
\end{proposition}
\begin{proof}
The functor $\ext\colon\poly\to\smset^\smset$ is fully faithful and hence reflects all colimits; by Proposition~\ref{prop.polycart_reflect}, $\polycart\to\poly$ does too. It suffices to show that $\polycart$ has coproducts and filtered colimits, because and that the two functors $\polycart\to\poly$ and $\polycart\to\smset^\smset$ preserve them.

Coproduct inclusions of polynomials are cartesian, so the coproduct cone exists in $\polycart$, and hence it is a coproduct. Since $\ext$ preserves coproducts, this concludes the case for coproducts.

By \cite{Adamek.Rosicky:1994a}, a category has (resp.\ a functor preserves) all filtered colimits iff it has (resp.\ preserves) all directed colimits, so let $\cat{I}$ be a directed category and $p\colon\cat{I}\to\polycart$ a directed sequence of polynomials and cartesian maps. Then the induced map $\cat{I}\to\poly$ has a colimit, say $\ol{p}$, and it is easy to check that all the structure maps $p_i\to\ol{p}$ are cartesian. Hence $\ol{p}\cong\colim_ip_i$ is a colimit in $\polycart$, and it is preserved by the inclusion.

It remains to check that $\ol{p}$ is a colimit in $\smset^\smset$. That is, we need to check that for any $X:\smset$, the function
\[
\sum_{P:\pi_0(\el p)}\colim_{e:\cat{E}_P}X^{p[e]}\too
\sum_{P:\pi_0(\el p)}X^{\lim_{e:\cat{E}_P}p[e]}
\]
is a bijection. Choose any $P$; it suffices to show that $\colim_{e:\cat{E}_P} X^{p[e]}\to X^{\lim_{e:\cat{E}_P}p[e]}$ is a bijection. But $\cat{E}_P$ is filtered and for any $e\to e'$ in $\cat{E}_P$, the map $p[e']\to p[e]$ is a bijection. Hence we can pick any object $e':\cat{E}_P$, replace all $p_e$ with $p_{e'}$, and have an isomorphic diagram. Now the colimit and limit are both constant, meaning that both sides of the desired map are isomorphic to $X^{p[e']}$, and hence to each other.
\end{proof}

\begin{restatable}{proposition}{colimleftdistribute}
\label{prop:colim-left-distribute}
For any polynomial $q$, filtered category $\cat{I}$, and diagram $p\colon\cat{I}\to\polycart$, the natural map
\[
\colim_{i:I}(p_i\tri q)\to(\colim_{i:I}p_i)\tri q
\]
is an isomorphism in $\poly$.
\end{restatable}
\begin{proof}\label{pf:colim-left-distribute}
By Lemma~\ref{lemma.ext_ff} it suffices to show that the function $\colim_{i:I}(p_i\tri q)\tri X\to(\colim_{i:I}p_i)\tri q\tri X$ is a bijection for any $X:\smset$. This is just two uses of Proposition~\ref{prop.polycart_colimits}:
\[
\colim_{i:I}(p_i\tri q)\tri X\cong\colim_{i:I}(p_i\tri q\tri X)\cong(\colim_{i:I}p_i)\tri q\tri X.
\qedhere
\]
\end{proof}

\section{Proofs for the Free Monad Monad}\label{apx:free-monad}

\associativity*
\begin{proof}
    We prove this lemma by transfinite induction on $\alpha$. First, for $\alpha = 0$ the result is immediate.

    Next, suppose that the result holds for $\alpha$, and we will show that it holds for the successor ordinal $\alpha + 1$. By definition of $p\hoc{\alpha + 1}$, it suffices to show that the outer square in the following diagram commutes. 

\[\begin{tikzcd}[column sep=75pt, row sep=large]
	{p\hoc{\beta} \tri p\hoc{\gamma} + p \tri p \hoc{\alpha} \tri p \hoc{\beta} \tri p \hoc{\gamma}} & {p \hoc{\beta + \gamma} + p \tri p \hoc{\alpha} \tri p \hoc{\beta + \gamma}} \\
	{p\hoc{\beta} \tri p \hoc{\gamma} + p \tri p \hoc{\alpha + \beta} \tri p \hoc{\gamma}} & {p\hoc{\beta + \gamma} + p \hoc {\alpha + \beta + \gamma}} \\
	{p\hoc{\alpha + \beta + 1} \tri p \hoc{\gamma}} & {p\hoc{\alpha + \beta + \gamma + 1}}
	\arrow["{\mu\hoc{\beta, \gamma} + p \tri p\hoc{\alpha} \tri \mu\hoc{\beta, \gamma}}", from=1-1, to=1-2]
	\arrow["{p\hoc{\beta} \tri p\hoc{\gamma} + p \tri \mu\hoc{\alpha, \beta}\tri p \hoc{\gamma}}"', from=1-1, to=2-1]
	\arrow["{(\iota\hoc{\beta}\tri p\hoc{\gamma}, \lambda\hoc{\alpha + \beta}\tri p\hoc{\gamma})}"', from=2-1, to=3-1]
	\arrow["{p\hoc{\beta + \gamma} + p \tri \mu\hoc{\alpha, \beta + \gamma}}", from=1-2, to=2-2]
	\arrow["{(\iota\hoc{\beta + \gamma}, \lambda\hoc{\alpha + \beta + \gamma})}", from=2-2, to=3-2]
	\arrow["{\mu\hoc{\beta, \gamma} + p \hoc{\alpha} \mu\hoc{\alpha + \beta, \gamma}}", from=2-1, to=2-2]
	\arrow["{\mu\hoc{\alpha + \beta + 1, \gamma}}"', from=3-1, to=3-2]
\end{tikzcd}\]

The top square commutes by the induction hypothesis. To check that the bottom squares commute, it suffices to show that the following diagrams commute as well.

\[\begin{tikzcd}[column sep=50pt, row sep = large]
	{p\hoc{\beta}\tri p\hoc{\gamma}} & {p\hoc{\beta + \gamma}} \\
	{p\hoc{\alpha + \beta + 1} \tri p \hoc{\gamma}} & {p\hoc{\alpha + \beta + \gamma + 1}}
	\arrow["{\mu\hoc{\beta, \gamma}}", from=1-1, to=1-2]
	\arrow["{\mu\hoc{\alpha + \beta + 1, \gamma}}"', from=2-1, to=2-2]
	\arrow["{\iota\hoc{\beta} \tri p \hoc{\gamma}}"', from=1-1, to=2-1]
	\arrow["{\iota\hoc{\beta + \gamma}}", from=1-2, to=2-2]
\end{tikzcd}\qqand
\begin{tikzcd}[column sep=50pt, row sep = large]
	{p \tri p\hoc{\alpha + \beta} \tri p \hoc{\gamma}} & {p \tri p \hoc{\alpha + \beta + \gamma}} \\
	{p \hoc{\alpha + \beta + 1} \tri p \hoc{\gamma}} & {p \hoc{\alpha + \beta + \gamma + 1}}
	\arrow["{p \tri \mu \hoc{\alpha + \beta, \gamma}}", from=1-1, to=1-2]
	\arrow["{\mu \hoc{\alpha + \beta + 1, \gamma}}"', from=2-1, to=2-2]
	\arrow["{\lambda\hoc{\alpha + \beta + \gamma}}", from=1-2, to=2-2]
	\arrow["{\lambda\hoc{\alpha + \beta} \tri p \hoc {\gamma}}"', from=1-1, to=2-1]
\end{tikzcd}\]
The diagram on the left commutes by Equation~\eqref{eq:p-alpha-beta}.
The diagram on the right commutes by the inductive definition of $\mu\hoc{\alpha + \beta + 1, \gamma}$.

Finally, suppose that $\alpha$ is a limit ordinal and suppose that the diagram commutes for all $\alpha' < \alpha$. Recall that $p\hoc \alpha = \colim_{\alpha' < \alpha} p \hoc{\alpha'}$. So it suffices to show that the outer diagram in the following diagram commutes. 

\[\begin{tikzcd}[sep=huge]
	{p\hoc{\alpha'} \tri p \hoc{\beta} \tri p \hoc {\gamma}} & {p\hoc{\alpha'} \tri p\hoc{\beta + \gamma}} & {\left(\colim_{\alpha' < \alpha} p \hoc{\alpha'}\right) \tri p\hoc{\beta + \gamma}} \\
	{p\hoc{\alpha' + \beta} \tri p \hoc{\gamma}} & {p \hoc{\alpha' + \beta + \gamma}} \\
	{\left(\colim_{\alpha' < \alpha } p \hoc{\alpha' + \beta}\right) \tri p \hoc {\gamma}} && {\colim_{\alpha' < \alpha} p \hoc{\alpha' + \beta + \gamma} }
	\arrow["{\mu\hoc{\alpha + \beta, \gamma}}"', from=3-1, to=3-3]
	\arrow["{\mu \hoc{\alpha, \beta + \gamma}}", from=1-3, to=3-3]
	\arrow[from=2-2, to=3-3]
	\arrow["{\mu \hoc{\alpha' + \beta, \gamma}}"', from=2-1, to=2-2]
	\arrow["{\mu\hoc{\alpha', \beta + \gamma}}", from=1-2, to=2-2]
	\arrow["{\mu\hoc{\alpha', \beta} \tri p \hoc{\gamma}}"', from=1-1, to=2-1]
	\arrow[from=2-1, to=3-1]
	\arrow[from=1-2, to=1-3]
	\arrow["{p \hoc{\alpha'} \tri \mu\hoc{\beta , \gamma}}", from=1-1, to=1-2]
\end{tikzcd}\]
The upper left square commutes by the induction hypothesis and the remaining diagrams commute by the definition of $\freem{p}$ for limit ordinals. 
\end{proof}

\freef*
\label{pf:freef}
\begin{proof}
    To show that $\free_f$ respects the identity it suffices to show that the outer diagram in the following commutes: 
\[\begin{tikzcd}
	\yon & {p\hoc{0}} & {\free_p} \\
	& {q\hoc{0}} & {\free_q}
	\arrow["\yon", from=1-1, to=1-2]
	\arrow["{f\hoc{0}}", from=1-2, to=2-2]
	\arrow["\yon"', from=1-1, to=2-2]
	\arrow["{\iota\hoc{0}}", from=1-2, to=1-3]
	\arrow["{\iota\hoc{0}}"', from=2-2, to=2-3]
	\arrow["{\free_f}", from=1-3, to=2-3]
\end{tikzcd}\]
The left triangle commutes because each map is the identity on $\yon$. The square commutes as it is the diagram in Equation~\eqref{eq:f-iota-commute}.

To show that $\free_f$ respects multiplication, it suffices to show that the following diagram commutes for all $\alpha$ and $\beta$.
\[\begin{tikzcd}
	{p\hoc{\alpha} \tri p \hoc{\beta}} & {p\hoc{\alpha + \beta}} \\
	{q\hoc{\alpha} \tri q\hoc{\beta}} & {q\hoc{\alpha + \beta}}
	\arrow["{\mu\hoc{\alpha, \beta}}", from=1-1, to=1-2]
	\arrow["{\mu\hoc{\alpha, \beta}}"', from=2-1, to=2-2]
	\arrow["{f\hoc{\alpha + \beta}}", from=1-2, to=2-2]
	\arrow["{f\hoc{\alpha} \tri f\hoc{\beta}}"', from=1-1, to=2-1]
\end{tikzcd}\]
We show this by induction on $\alpha$. For $\alpha = 0$ this diagram commutes because the horizontal maps are the identity and the vertical maps are $f\hoc{\beta}$. Suppose that the diagram commutes for $\alpha$. Then for the successor ordinal $\alpha + 1$, following the definitions of $\mu\hoc{\alpha + 1, \beta}$ and $f\hoc{\alpha + 1, \beta}$ we want to show that the following diagram commutes.
\[\begin{tikzcd}[column sep = 9em, row sep = large]
	{ p \hoc{\beta} + p \tri p \hoc{\alpha} \tri p \hoc{\beta}} & {p\hoc{\alpha + \beta + 1}} \\
	{ q \hoc{\beta} + q \tri q \hoc{\alpha} \tri q \hoc{\beta}} & {q\hoc{\alpha + \beta + 1}}
	\arrow["{(\iota\hoc{\beta},(p \tri \mu \hoc{\alpha, \beta}) \then \lambda\hoc{\alpha + \beta})}", from=1-1, to=1-2]
	\arrow["{(\iota\hoc{\beta}, (p \tri \mu \hoc{\alpha, \beta}) \then \lambda\hoc{\alpha + \beta})}"', from=2-1, to=2-2]
	\arrow["{f\hoc{\alpha + \beta + 1}}", from=1-2, to=2-2]
	\arrow["{f\hoc{\beta} + f \tri f\hoc{\alpha} \tri f\hoc{\beta}}"', from=1-1, to=2-1]
\end{tikzcd}\]
The diagram commutes on the first term of the coproduct because the diagram in Equation~\eqref{eq:f-iota-commute} commutes with the pair $\beta < \alpha + \beta$. To show that the diagram commutes on the second term, it suffices to show that the outer diagram in the following commutes.
\[\begin{tikzcd}[sep=huge]
	{p \tri p \hoc{\alpha} \tri p \hoc{\beta}} & {p \tri p \hoc{\alpha + \beta}} & {p \hoc{\alpha + \beta + 1}} \\
	{q\tri q\hoc{\alpha} \tri q \hoc{\beta}} & {q\tri q\hoc{\alpha + \beta}} & {q\hoc{\alpha + \beta + 1}}
	\arrow["{f \tri f\hoc{\alpha} \tri f\hoc{\beta}}"', from=1-1, to=2-1]
	\arrow["{p \tri \mu\hoc{\alpha, \beta}}", from=1-1, to=1-2]
	\arrow["{q\tri \mu\hoc{\alpha, \beta}}"', from=2-1, to=2-2]
	\arrow["{f \tri f\hoc{\alpha + \beta}}", from=1-2, to=2-2]
	\arrow["{f\hoc{\alpha + \beta + 1}}", from=1-3, to=2-3]
	\arrow["{\lambda\hoc{\alpha + \beta}}", from=1-2, to=1-3]
	\arrow["{\lambda \hoc{\alpha + \beta}}"', from=2-2, to=2-3]
\end{tikzcd}\]
The left-hand square commutes by the induction hypothesis and the right-hand square commutes by definition of $f\hoc{\alpha + \beta + 1}$.
\end{proof}

\subsection{Proof of Theorem~\ref{thm:free-adjunction}}
\label{apx:adjunction}
The trickiest part of proving the adjunction in Theorem~\ref{thm:free-adjunction} is defining the co-unit so we do it separately in the following series of Lemmas.

\newcommand{\cou}[2]{\epsilon\hoc{#2}}  
\newcommand{\couq}[1]{\cou{q}{#1}}      
\begin{lemma}
    Let $(q, \eta_q, \mu_q)$ be a $\tri$-monoid. For each ordinal $\alpha$, there exist a cocone of maps
    \[
        \couq{\alpha} \colon q\hoc{\alpha} \to q.
    \]
\end{lemma}
\begin{proof}
    We will define the cocone of maps $\couq{\alpha}$ inductively.
    First, define $\couq{0} \colon q \hoc{0} = \yon \to q$ to be the unit $\eta_q$.

    Next, suppose that we have a cocone of maps $\couq{\alpha'}: q\hoc{\alpha'} \to q$ for all $\alpha' < \alpha$. 
    If $\alpha$ is a limit ordinal then we define $\couq{\alpha}\colon q\hoc{\alpha} \to q$ to be the universal map induced by the cocone. If $\alpha + 1$ is a successor ordinal, then we define $\couq{\alpha + 1}$ to be the composite
    \[
        q\hoc{\alpha + 1} = \yon + q \tri q \hoc{\alpha} \xrightarrow{\eta_q + q \tri \couq{\alpha}} q + q \tri q \xrightarrow{(1, \mu_q)}q.
    \]

    To show that these maps form a cocone, it suffices to show that the following diagram commutes for all $\alpha' < \alpha$.
    \[\begin{tikzcd}
    	{q\hoc{\alpha'}} \\
    	{q\hoc{\alpha}} & q
    	\arrow["{\couq{\alpha'}}", from=1-1, to=2-2]
    	\arrow["{\couq{\alpha}}"', from=2-1, to=2-2]
    	\arrow["{\iota\hoc{\alpha'}}"', from=1-1, to=2-1]
    \end{tikzcd}\]

    If $\alpha$ is a limit ordinal, then this is immediate. Otherwise, $\alpha$ is a successor ordinal, say $\alpha = \beta + 1$. It suffices to show that the following diagram commutes. 
    \[\begin{tikzcd}
    	{q\hoc{\beta}} \\
    	{q\hoc{\beta + 1}} & q
    	\arrow["{\couq{\beta}}", from=1-1, to=2-2]
    	\arrow["{\couq{\beta + 1}}"', from=2-1, to=2-2]
    	\arrow["{\iota\hoc{\beta}}"', from=1-1, to=2-1]
    \end{tikzcd}\]
    
    Consider the following three cases.
    \begin{itemize}
        \item Suppose that $\beta = 0$, then it is immediate that the desired diagram commutes as we see below.
        \[\begin{tikzcd}[column sep = 60pt]
        	{q\hoc{0} = \yon} \\
        	{q\hoc{1} = \yon + q \tri \yon} & {q + q} & q
        	\arrow["{\iota\hoc{0}}"', from=1-1, to=2-1]
        	\arrow["{(1,1)}"', from=2-2, to=2-3]
        	\arrow["{\eta_q + \mu_q \circ(q \tri \eta_q)}"', from=2-1, to=2-2]
        	\arrow["{\eta_q}", bend left=10pt, from=1-1, to=2-3]
        \end{tikzcd}\]

        \item Next suppose that $\beta$ is a successor ordinal. In particular $\beta = \gamma + 1$. 
        To show that 
        \[\begin{tikzcd}[sep = large]
        	{q\hoc{\gamma + 1}} \\
        	{q\hoc{\gamma + 2}} & q
        	\arrow["{\iota\hoc{\gamma + 1}}"', from=1-1, to=2-1]
        	\arrow["{\couq{\gamma + 1}}", from=1-1, to=2-2]
        	\arrow["{\couq{\gamma + 2}}"', from=2-1, to=2-2]
        \end{tikzcd}\]
        commutes, it suffices to show that 
        \[\begin{tikzcd}[sep = huge]
        	{\yon + q \tri q \hoc{\gamma}} \\
        	{\yon + q \tri q \hoc{\gamma + 1}} & {q + q \tri q}
        	\arrow["{\eta_q + q \tri \couq{\gamma + 1}}"', from=2-1, to=2-2]
        	\arrow["{\yon + q \tri \iota\hoc{\gamma}}"', from=1-1, to=2-1]
        	\arrow["{\eta_q + q \tri \couq{\gamma}}", from=1-1, to=2-2]
        \end{tikzcd}\]
        commutes. This is immediate from the induction hypothesis. 
        \item Finally, suppose that $\beta$ is a limit ordinal. Then we want to show that the outer diagram in the following commutes. 

        \[\begin{tikzcd}[sep = huge]
        	{q\hoc{\beta'}} \\
        	{q\hoc{\beta' + 1} = \yon + q \tri q\hoc{\beta'}} \\
        	{\yon + q \tri q\hoc{\beta}} & {q + q\tri q} & q
        	\arrow["{\iota\hoc{\beta'}}"', from=1-1, to=2-1]
        	\arrow["{\yon + q \tri \iota\hoc{\beta'}}"', from=2-1, to=3-1]
        	\arrow["{(1, \mu_q)}"', from=3-2, to=3-3]
        	\arrow["{\couq{\beta'}}", from=1-1, to=3-3]
        	\arrow["{\eta_q + q \tri \couq{\beta}}"', from=3-1, to=3-2]
        	\arrow["{\couq{\beta' + 1}}"{description}, from=2-1, to=3-3]
        	\arrow["{\eta_q + q \tri \couq{\beta'}}"{description}, from=2-1, to=3-2]
        \end{tikzcd}\]
        The top and bottom triangles commute by the induction hypothesis and the middle triangle commutes by definition of $\couq{\beta' + 1}$.
    \end{itemize}

\end{proof}

\begin{lemma}\label{lem:cou-mu-commute}
    For all ordinals $\alpha, \beta$ the following diagram commutes.
    \[\begin{tikzcd}
    	{q\hoc{\alpha} \tri q\hoc{\beta}} & {q\hoc{\alpha + \beta}} \\
    	{q\tri q} & q
    	\arrow["{\mu_q}"', from=2-1, to=2-2]
    	\arrow["{\couq{\alpha + \beta}}", from=1-2, to=2-2]
    	\arrow["{\couq{\alpha} \tri \couq{\beta}}"', from=1-1, to=2-1]
    	\arrow["{\mu\hoc{\alpha, \beta}}", from=1-1, to=1-2]
    \end{tikzcd}\]
\end{lemma}
\begin{proof}
    We induct on $\alpha$.  Suppose that for all $\alpha' < \alpha$ the diagram commutes. Consider the following cases.

    \begin{itemize}
        \item For $\alpha = 0$, we have that 
        \[\begin{tikzcd}
        	{\yon \tri q\hoc{\beta}} & {q\hoc{\beta}} \\
        	{q\tri q} & q
        	\arrow["{\mu_q}"', from=2-1, to=2-2]
        	\arrow["{1_{q\hoc{\beta}}}", from=1-1, to=1-2]
        	\arrow["{\eta_q \tri \couq{\beta}}"', from=1-1, to=2-1]
        	\arrow["{\couq{\beta}}", from=1-2, to=2-2]
        \end{tikzcd}\]
        commutes by the unit law of the monoid $(q, \eta_q, \mu_q)$.

        \item Suppose that $\alpha$ is a successor ordinal, say $\alpha = \alpha' + 1$. We must show that the following diagram commutes. 
        \[\begin{tikzcd}[column sep = 9em, row sep = huge]
        	{\yon \tri q\hoc{\beta} + q \tri q\hoc{\alpha'} \tri q\hoc{\beta}} & {q\hoc{\alpha' + \beta + 1}} \\
        	{q\tri q} & q
        	\arrow["{(\eta \tri \couq{\beta}, (\mu _q \circ(q \tri \couq{\alpha'})) \tri \couq{\beta})}"', from=1-1, to=2-1]
        	\arrow["{(\iota_\beta, \lambda\hoc{\alpha' + \beta} \circ (q \tri \mu\hoc{\alpha', \beta}))}", from=1-1, to=1-2]
        	\arrow["{\couq{\alpha' + \beta + 1}}", from=1-2, to=2-2]
        	\arrow["{\mu_q}"', from=2-1, to=2-2]
        \end{tikzcd}\]
        We will show that the diagram commutes on each component of the coproduct independently. For the first term observe  that 
        \[\begin{tikzcd}
        	{\yon \tri q\hoc{\beta} } & {q\hoc{\beta}} & {q\hoc{\alpha' + \beta + 1}} \\
        	{q\tri q} & q
        	\arrow["{\eta \tri \couq{\beta}}"', from=1-1, to=2-1]
        	\arrow["{1_{q\hoc\beta}}", from=1-1, to=1-2]
        	\arrow["{\iota_\beta}", from=1-2, to=1-3]
        	\arrow["{\couq{\beta}}", from=1-2, to=2-2]
        	\arrow["{\mu_q}"', from=2-1, to=2-2]
        	\arrow["{\couq{\alpha' + \beta + 1}}", from=1-3, to=2-2]
        \end{tikzcd}\]
        commutes because of the unit law of the monoid $(q, \eta_q, \mu_q)$ and because the maps $\couq{\beta}$ form a cocone.

        For the second term consider the following diagram.

        \[\begin{tikzcd}[sep = huge]
        	{q\tri q\hoc{\alpha'} \tri q \hoc{\beta}} & {q \tri q\hoc{\alpha' + \beta}} & {q\hoc{\alpha' + \beta + 1}} \\
        	{q\tri q \tri q} & {q\tri q} \\
        	{q\tri q} && q
        	\arrow["{q\tri \couq{\alpha'} \tri \couq{\beta}}"', from=1-1, to=2-1]
        	\arrow["{\mu_q \tri q}"', from=2-1, to=3-1]
        	\arrow["{q \tri \mu_q}"', from=2-1, to=2-2]
        	\arrow["{\mu_q}"', from=2-2, to=3-3]
        	\arrow["{\mu_q}"', from=3-1, to=3-3]
        	\arrow["{q\tri \mu\hoc{\alpha', \beta}}", from=1-1, to=1-2]
        	\arrow["{q\tri \couq{\alpha' + \beta}}", from=1-2, to=2-2]
        	\arrow["{\lambda\hoc{\alpha' + \beta}}", from=1-2, to=1-3]
        	\arrow["{\couq{\alpha' + \beta + 1}}", from=1-3, to=3-3]
        \end{tikzcd}\]
        The upper left square commutes by the induction hypothesis. The bottom square commutes by associativity of $\mu_q$, and the right-most square commutes by definition of $\couq{\alpha' + \beta + 1}$. Therefore the outer square commutes as desired.

        \item Suppose that $\alpha$ is a limit ordinal. Recall that 
        \[
            q\hoc{\alpha} \tri q\hoc{\beta} \iso \colim_{\alpha' < \alpha} q\hoc{\alpha' + \beta}.
        \] 
        The composite $\couq{\alpha + \beta} \circ \mu\hoc{\alpha, \beta}$ is induced by the cocone defined for $\alpha' < \alpha$ by
        \[
            q\hoc{\alpha'} \tri q \hoc{\beta} \xrightarrow{\mu\hoc{\alpha', \beta}} q\hoc{\alpha' + \beta}\xrightarrow{\iota\hoc{\alpha' + \beta}} q \hoc{\alpha + \beta} \xrightarrow{\couq{\alpha + \beta}} q
        \] while the composite $\mu_q \circ (\couq{\alpha} \tri \couq{\beta})$ is induced by the cocone defined for $\alpha' < \alpha$ by 
        \[
            q\hoc{\alpha'} \tri q \hoc{\beta }\xrightarrow{\couq{\alpha'} \tri \couq{\beta}} q \tri q \xrightarrow{\mu_q}q.
        \] For each $\alpha' < \alpha$ these composites are identical by the induction hypothesis. Therefore by uniqueness of the induced maps \[\couq{\alpha + \beta} \circ \mu\hoc{\alpha, \beta} = \mu_q \circ (\couq{\alpha} \tri \couq{\beta})\] as desired.
    \end{itemize}
\end{proof}

\newcommand{\mcoun}{\epsilon}
\newcommand{\mun}{\zeta}

\freeadjunction*
\begin{proof}
    For $p: \poly$ define the unit of the adjunction, $\mun_p \colon p \to \free_p$, to be the composite
    \[
        p \xrightarrow{\lambda\hoc{0}} p\hoc{1} \xrightarrow{\iota\hoc{1}} \free_p.
    \]
    For $f\colon p \to q$ in $\poly$, consider the following diagram.
    \[\begin{tikzcd}
    	p & {\yon + p  = p\hoc{1}} & {\free_p} \\
    	q & {\yon + q = q\hoc{1}} & {\free_q}
    	\arrow["f"', from=1-1, to=2-1]
    	\arrow["{\yon + f = f\hoc{1}}", from=1-2, to=2-2]
    	\arrow["{\lambda\hoc{0}}", from=1-1, to=1-2]
    	\arrow["{\lambda\hoc{0}}"', from=2-1, to=2-2]
    	\arrow["{\free_f}", from=1-3, to=2-3]
    	\arrow["{\iota\hoc{1}}"', from=2-2, to=2-3]
    	\arrow["{\iota\hoc{1}}", from=1-2, to=1-3]
    \end{tikzcd}\]
    The square on the left commutes by definition of $f\hoc{1}$ while the square on the right commutes by the commuting diagram in Equation~\eqref{eq:f-iota-commute}. Therefore, the outer square commutes and $\mun$ is natural.

    For the  $\tri$-monoid $(q, \eta_q, \mu_q)$, we define the counit of the adjunction $\mcoun_q \colon (\free_q, \freeu{q}, \freem{q}) \to (q, \eta_q, \mu_q)$ to be $\epsilon\hoc{\kappa} \colon \free_q \to q$ for $\kappa$ such that $q$ is $\kappa$-small. 
    The map $\mcoun_q$ preserves the unit because the maps $\couq{\alpha}$ form a cocone. 
    That $\mcoun_q$ preserves multiplication is a direct result of Lemma~\ref{lem:cou-mu-commute}. Therefore $\mcoun_q$ is indeed a map of $\tri$-monoids. 

    Next we want to show that $\mcoun$ is natural. Let $f\colon (p, \eta_p, \mu_p) \to (q, \eta_q, \mu_q)$ in $\Mod(\poly)$. It suffices to show for all $\alpha$ the following diagram commutes. 
    \[\begin{tikzcd}
    	{p\hoc{\alpha}} & {q\hoc{\alpha}} \\
    	p & q
    	\arrow["{\cou{p}{\alpha}}"', from=1-1, to=2-1]
    	\arrow["{\couq{\alpha}}", from=1-2, to=2-2]
    	\arrow["{f\hoc{\alpha}}", from=1-1, to=1-2]
    	\arrow["f"', from=2-1, to=2-2]
    \end{tikzcd}\]

    We show this by induction on $\alpha$. For $\alpha = 0$, $f\hoc{\alpha}$ is the identity on $\yon$ and the diagram commutes because $f$ preserves the unit.

    Suppose that the diagram commutes for all $\alpha' < \alpha$. If $\alpha$ is a successor ordinal --- say $\alpha = \alpha' + 1$ --- then we want to show that the following diagram commutes.
    \[\begin{tikzcd}
    	{\yon + p \tri p\hoc{\alpha'}} & {\yon + q \tri q\hoc{\alpha'}} \\
    	p & q
    	\arrow["{\yon + f \tri f\hoc{\alpha'}}", from=1-1, to=1-2]
    	\arrow["f", from=2-1, to=2-2]
    	\arrow["{(\eta_p, \mu_p \circ (p \tri \cou{p}{\alpha'}))}"', from=1-1, to=2-1]
    	\arrow["{(\eta_q, \mu_q\circ(q \tri \couq{\alpha'})) }", from=1-2, to=2-2]
    \end{tikzcd}\]

    It commutes on the first term of the coproduct, again because $f$ preserves the unit. To show that it commutes on the second term, we want to show that the outer diagram in the following commutes. 

    \[\begin{tikzcd}[column sep = large]
    	{p \tri p \hoc{\alpha'}} & {q\tri q\hoc{\alpha'}} \\
    	{p \tri p} & {q\tri q} \\
    	p & q
    	\arrow["{f\tri f\hoc{\alpha'}}", from=1-1, to=1-2]
    	\arrow["{p \tri \cou{p}{\alpha'}}"', from=1-1, to=2-1]
    	\arrow["{q\tri \couq{\alpha'}}", from=1-2, to=2-2]
    	\arrow["{f\tri f}"', from=2-1, to=2-2]
    	\arrow["{\mu_p}"', from=2-1, to=3-1]
    	\arrow["{\mu_q}", from=2-2, to=3-2]
    	\arrow["f", from=3-1, to=3-2]
    \end{tikzcd}\]
    The top square commutes by the induction hypothesis and the bottom square commutes because $f$ preserves multiplication.


If $\alpha$ is a limit ordinal, consider the cocones
\[
    p\hoc{\alpha'} \To{f\hoc{\alpha'}} q\hoc{\alpha'} \To{\epsilon\hoc{\alpha'}} q
    \qqand
    p\hoc{\alpha'} \To{\epsilon\hoc{\alpha'}} p \To{f} q.
\]
By uniqueness, the cocone on the left induces the map $p\hoc{\alpha} \To{f\hoc{\alpha}} q\hoc{\alpha} \To{\epsilon\hoc{\alpha}} q$. Also by uniqueness, the cocone on the right induces the map $p\hoc{\alpha} \To{\epsilon\hoc{\alpha}} p \To{f} q$. Furthermore, these cocones are identical by the induction hypothesis and so their induced maps are equal, as desired.

Finally, to show that the unit and counit form an adjunction we must show that for a polynomial $p$
\[
    \free_p \xrightarrow{\free_{\mun_p}} \free_{\free_p} \xrightarrow{\mcoun_{\free_p}} \free_p
\] is the identity. 

It suffices to show that for all $\alpha$, \[p\hoc{\alpha} \xrightarrow{(\mun_p)\hoc{\alpha}} (\free_p)\hoc{\alpha} \xrightarrow{\mcoun\hoc{\alpha}} \free_p\] is the inclusion $\iota\hoc{\alpha}$. We do this by transfinite induction. The base case and the induction step for limit ordinals follow directly from definitions. For successor ordinals $\alpha + 1$, it suffices to show that the following diagram commutes:

\[\begin{tikzcd}[column sep=huge]
	{\yon + p \tri p\hoc{\alpha}} && {\yon+ p \tri \free_p} \\
	{\yon + p \tri p \hoc{\alpha}} & {\yon + p \tri (\free_p)\hoc{\alpha}} & {\yon + p \tri \free_p} \\
	{p\hoc{\alpha+1}} & {\yon + \free_p \tri (\free_p)\hoc{\alpha}} & {\free_p}
	\arrow["{\yon + p \tri \iota\hoc{\alpha}}", from=1-1, to=1-3]
	\arrow[Rightarrow, no head, from=1-1, to=2-1]
	\arrow[Rightarrow, no head, from=1-3, to=2-3]
	\arrow["{\yon + p \tri (\mun_p)\hoc{\alpha}}", from=2-1, to=2-2]
	\arrow[Rightarrow, no head, from=2-1, to=3-1]
	\arrow["{\yon + p \tri \mcoun\hoc{\alpha}}", from=2-2, to=2-3]
	\arrow["{\yon + \mun_p \tri (\free_p)\hoc{\alpha}}", from=2-2, to=3-2]
	\arrow["{(\mun_p)\hoc{\alpha+1}}"', from=3-1, to=3-2]
	\arrow["{\mcoun\hoc{\alpha+1}}"', from=3-2, to=3-3]
	\arrow["\iso"', from=3-3, to=2-3]
\end{tikzcd}\]

The top square commutes by the induction hypothesis. The bottom left and right squares commute by definitions of $(\mun_p)\hoc{\alpha+1}$ and $\mcoun\hoc{\alpha+1}$, respectively.

We must also show that for a $\tri$-monoid $(q, \eta_q, \mu_q)$
\[
    q \xrightarrow{\mun_q} \free_q \xrightarrow{\mcoun_q} q
\] is the identity as well. 

Since the diagrams below commutes, it suffices to show that $\couq{1} \circ \lambda\hoc{0}$ is the identity. This is immediate from the definition of $\couq{1}$ and the unit law for $(q, \eta_q, \mu_q)$.

\[\begin{tikzcd}
	q & {q\hoc{1}} & {\free_q} & q
	\arrow["{\epsilon_q}", from=1-3, to=1-4]
	\arrow["{\iota\hoc{1}}", from=1-2, to=1-3]
	\arrow["{\lambda\hoc{0}}", from=1-1, to=1-2]
	\arrow["{\epsilon\hoc{1}}"', curve={height=18pt}, from=1-2, to=1-4]
\end{tikzcd}\]    
\end{proof}

\section{Proofs for the Cofree Comonad Monad}\label{app.cofree}

\propcofree*

\begin{proof}
Given a polynomial $p$, define polynomials $p^{(i)}$ for $i\in\nn$ by
\[
  p\coh{0}\coloneqq\yon
  \qqand
  p\coh{1+i}\coloneqq\yon\times\left(p\tri p\coh{i}\right)
\]
There is a projection map $\pi\coh{0}\colon p\coh{1}\to p\coh{0}$, and if $\pi\coh{i}\colon p\coh{1+i}\to p\coh{i}$ has been defined, then we can define $\pi\coh{1+i}\coloneqq \yon\times(p\tri\pi\coh{i})$. Now define the polynomial
\begin{equation}\label{eqn.construct_cofree}
\cofree_p\coloneqq\lim\big(\cdots\To{\pi\coh{2}}p\coh{2}\To{\pi\coh{1}}p\coh{1}\To{\pi\coh{0}}p\coh{0}\big)
\end{equation}
and we note that this construction $p\mapsto \cofree_p$ is natural in $p:\poly$.

This polynomial comes equipped with a counit $\epsilon\colon\cofree_p\to\yon=p\coh{0}$ given by the projection. We next construct the comultiplication $\delta\colon\cofree_p\to\cofree_p\tri\cofree_p$. Since $\tri$ commutes with connected limits, we have
\[
  \cofree_p\tri\cofree_p=
  \left(\lim_{i_1}p\coh{i_1}\right)\tri\left(\lim_{i_2}p\coh{i_2}\right)\cong
  \lim_{i_1,i_2}\left(p\coh{i_1}\tri p\coh{i_2}\right)
\]
To obtain the comultiplication $\lim_ip\coh{i}\to\lim_{i_1,i_2}(p\coh{i_1}\tri p\coh{i_2})$, it suffices to produce a natural choice of polynomial map $\varphi_{i_1,i_2}\colon p\coh{i_1+i_2}\to p\coh{i_1}\tri p\coh{i_2}$ for any $i_1,i_2:\nn$. When $i_1=0$ or $i_2=0$, we use the unit identity for $\tri$. By induction, assume given $\varphi_{i_1,1+i_2}$; we construct $\varphi_{1+i_1,1+i_2}$ as follows:
\begin{align}
\nonumber
  p\coh{1+i_1+1+i_2}&=
  \yon\times \left(p\tri p\coh{i_1+1+i_2}\right)\\&\to
\label{eqn.induction}
  \yon\times \left(p\tri p\coh{i_1}\tri p\coh{1+i_2}\right)\\&\to
\label{eqn.special}
  \left(\yon\times p\tri p\coh{i_1}\right)\tri p\coh{1+i_2}\\&=
\nonumber
  p\coh{1+i_1}\tri p\coh{1+i_2}
\end{align}
where \eqref{eqn.induction} is $\varphi_{i_1,1+i_2}$ and it remains to construct \eqref{eqn.special}. Since $-\tri q$ preserves products for any $q$, constructing \eqref{eqn.special} is equivalent to constructing two maps
\[
\yon\times \left(p\tri p\coh{i_1}\tri p\coh{1+i_2}\right)\To{\phi\coh{i_1,i_2}} p\coh{1+i_2}
\qqand
\yon\times \left(p\tri p\coh{i_1}\tri p\coh{1+i_2}\right)\to p\tri p\coh{i_1}\tri p\coh{1+i_2}.
\]
For the latter we use the second projection. The former, $\phi\coh{i_1,i_2}\colon p\coh{1+i_1+1+i_2}\to p\coh{1+i_2}$, is the more interesting one; for it we also use projections $p\coh{i_1}\to p\coh{0}=\yon$ and $\pi\coh{i_2}\colon p\coh{i_2+1}\to p\coh{i_2}$ to obtain:
\[
\yon\times \left(p\tri p\coh{i_1}\tri p\coh{1+i_2}\right)\to
\yon\times \left(p\tri\yon\tri p\coh{i_2}\right)\cong p\coh{1+i_2}
\]
We leave the naturality of this to the reader.

It remains to check that $\epsilon$ and $\delta$ satisfy unitality and coassociativity. The base cases above imply unitality. Proving coassociativity amounts to proving that the following diagram commutes:
\[
\begin{tikzcd}[column sep=50pt]
	p\coh{1+i_1+1+i_2+1+i_3}\ar[r, "\phi\coh{i_1,i_2+1+i_3}"]\ar[d, "\phi\coh{i_1+1+i_2,i_3}"']&
	p\coh{1+i_2+1+i_3}\ar[d, "\phi\coh{i_2,i_3}"]\\
	p\coh{1+i_3}\ar[r,equal]&p\coh{1+i_3}
\end{tikzcd}
\]
This can be shown by induction on $i_3$.
\end{proof}

\begin{proposition}
\label{prop:cofree-monoidal}
There is a monoidal structure on $\cofree\colon\poly\to\poly$
\[
  \yon\to\cofree_\yon
  \qqand
  \cofree_p\otimes\cofree_q\to\cofree_{p\otimes q}.
\]
\end{proposition}
\begin{proof}
The polynomial $\cofree_\yon\cong\yon^\nn$ has a unique position, and this defines the first map. However, it is conceptually cleaner to realize that comonads are closed under $\otimes$ by duoidality, and hence both $\yon$ and $\cofree_p\otimes\cofree_q$ carry comonad structures. Thus the desired maps are induced by the obvious polynomial maps $\yon\cong\yon$ and $\cofree_p\otimes\cofree_q\to p\otimes q$. It is straightforward to check that these are unital and associative.
\end{proof}

\cofreeadjunction*

\begin{proof}
We will abuse notation and denote the comonoid $(c,\epsilon,\delta):\catsharp$ simply by its carrier $c$. We first provide the counit and unit of the desired adjunction. The counit
\[
\epsilon_p\colon\cofree_p\to p
\]
is given by composing the projection map $\cofree_p\to p\coh{1}$ from construction \eqref{eqn.construct_cofree} with the projection $p\coh{1}\cong\yon\times p\to p$. Since $\cofree_c$ is defined as a limit, the unit
\[
\eta_c\colon c\coto\cofree_c
\]
will be given by defining maps $\eta\coh{i}\colon c\to c\coh{i}$ commuting with the projections $\pi\coh{i}\colon c\coh{1+i}\to c\coh{i}$, for each $i:\nn$, and then showing that the resulting polynomial map $\eta_c$ is indeed a cofunctor. Noting that $c\coh{0}=\yon$, we define
\[
\eta\coh{0}\coloneqq\epsilon
\]
Given $\eta\coh{i}\colon c\to c\coh{i}$, we define $\eta\coh{1+i}$ as the composite
\[
c\To{(\epsilon,\delta)}\yon\times(c\tri c)\To{\yon\times(c\tri\eta\coh{i})}\yon\times\left(c\tri c\coh{i}\right)=c\coh{1+i}.
\]
Clearly, we have $\eta\coh{0}=\pi\coh{0}\circ\eta\coh{1}$. It is easy to check that if $\eta\coh{i}=\pi\coh{i}\circ\eta\coh{1+i}$ then $\eta\coh{1+i}=\pi\coh{1+i}\circ\eta\coh{2+i}$. Thus we have constructed a polynomial map $\eta\colon c\to \cofree_c$. It clearly commutes with the counit, so it suffices to show that $\eta$ commutes with the comultiplication, which amounts to showing that the following diagram commutes
\[
\begin{tikzcd}
  c\ar[r, "\delta"]\ar[d, "{(\epsilon,\delta)}"']&[27pt]
  c\tri c\ar[r, "{(\epsilon,\delta)\tri(\epsilon,\delta)}"]&[15pt]
  \big(\yon\times(c\tri c)\big)\tri\big(\yon\times(c\tri c)\big)\ar[d, "{(\yon\times c\tri\eta\coh{i_1})\tri(\yon\times c\tri\eta\coh{i_2})}"]\\
  \yon\times(c\tri c)\ar[r, "{\yon\times c\tri\eta\coh{i_1+1+i_2}}"']&
  \yon\times(c\tri c\coh{i_1+1+i_2})\ar[r, "\varphi_{1+i_1,1+i_2}"']&
  \big(\yon\times(c\tri c\coh{i_1})\big)\tri\big(\yon\times(c\tri c\coh{i_2})\big)
\end{tikzcd}
\]
for all $i_1,i_2:\nn$, where $\varphi_{1+i_1,1+i_2}$ is the map constructed in Equation~\eqref{eqn.induction} and Equation~\eqref{eqn.special}. Commutativity follows from the counitality and coassociativity of the comonoid $c$.

The triangle identities are straightforward as well. Indeed, for any comonoid $c:\catsharp$, the composite $c\To{U\circ\eta_c} \cofree_c\To{\epsilon_{Uc}} c$ is equal to the composite of $c\To{(\epsilon,c)}c\coh{1}=\yon\times c$, with the projection $c\coh{1}\to c$, the result of which is the identity. Finally, for any polynomial $p:\poly$, the composite $\cofree_p\To{\eta_{\cofree_p}}\cofree_{\cofree_p}\To{\cofree_{\epsilon_p}}\cofree_p$ is given by taking a limit of maps of the form
\[
	\cofree_p\To{(\epsilon,\delta)}
	\yon\times(\cofree_p\tri\cofree_p)\To{\yon\times(\cofree_p\tri\eta\coh{i})}
	\yon\times(\cofree_p\tri\cofree_p\coh{i})\To{\yon\times(\epsilon_p\tri\epsilon_p\coh{i})}
	\yon\times(p\tri p\coh{i})
\]
Each one is in fact the projection $\cofree_p\to p\coh{i+1}$, so the resulting map is the identity on $\cofree_p$, completing the proof.
\end{proof}

\section{Proofs for the Module Structure $\free_p \otimes \cofree_q \to \free_{p \otimes q}$}
\label{apx:module}
\naturalinteraction*\label{pf:natural-interaction}
\begin{proof}
    To show that $\psi$ is natural in $p$, it suffices to show that for a maps $ p \to p'$ and $q \to q'$ in $\poly$ the following diagram commutes:
    \[\begin{tikzcd}
    	p & {[\cofree_q, \free_{p \otimes q}]} \\
    	{p'} & {[\cofree_{q'}, \free_{p' \otimes q'}]}
    	\arrow[from=1-1, to=2-1]
    	\arrow[from=1-1, to=1-2]
    	\arrow[from=1-2, to=2-2]
    	\arrow[from=2-1, to=2-2]
    \end{tikzcd}\]
    This follows immediately from the commutativity of the following diagram.
    \[\begin{tikzcd}
    	{p \otimes \cofree_q} & {p \otimes q} & {\free_{p \otimes q}} \\
    	{p' \otimes \cofree_{q'}} & {p'\otimes q'} & {\free_{p'\otimes q'}}
    	\arrow[from=1-1, to=2-1]
    	\arrow[from=1-1, to=1-2]
    	\arrow[from=1-2, to=1-3]
    	\arrow[from=2-1, to=2-2]
    	\arrow[from=1-2, to=2-2]
    	\arrow[from=2-2, to=2-3]
    	\arrow[from=1-3, to=2-3]
    \end{tikzcd}
    \]
    Note that the square on the left commutes by naturality of the counit of the adjunction in \cref{thm:cofree_comonad_comonad}. The square on the right commutes by naturality of unit of the adjunction in \cref{thm:free-adjunction}.
\end{proof}

\module* \label{pf:module}

\begin{proof}
    We must show that two diagrams commute. First, we will  show that the following diagram commutes.
    \[\begin{tikzcd}
    	{\yon \otimes \free_r} & {\free_r} \\
    	{\cofree_\yon \otimes \free_r} & {\free_{\yon \otimes r}}
    	\arrow["\iso", from=1-1, to=1-2]
    	\arrow[from=1-1, to=2-1]
    	\arrow[from=2-1, to=2-2, "\modstruct_{\yon,r}"']
    	\arrow["\iso"', from=2-2, to=1-2]
    \end{tikzcd}\]
    It suffices to show that the diagram below commutes: 
    \[\begin{tikzcd}
    	{\yon \otimes r} & r & {\free_r} \\
    	{\cofree_\yon \otimes r} & {\yon \otimes r} & {\free_{\yon \otimes r}}
    	\arrow[from=2-3, to=1-3]
    	\arrow[from=1-1, to=2-1]
    	\arrow[from=2-1, to=2-2]
    	\arrow[from=2-2, to=2-3]
    	\arrow[from=1-1, to=1-2]
    	\arrow[from=1-2, to=1-3]
    	\arrow[from=2-2, to=1-2]
    \end{tikzcd}\]
    The square on the left commutes because $\yon \to \cofree_\yon \to \yon$ is the identity. The square on the right commutes by the naturality of the unit of the adjunction in Theorem~\ref{thm:free-adjunction}.

    Next we must show that the following diagram commutes.
    \[\begin{tikzcd}
    	{\cofree_p \otimes \cofree_q \otimes \free_r} & {\cofree_p \otimes \free_{q \otimes r}} \\
    	{\cofree_{p \otimes q} \otimes \free_r} & {\free_{p \otimes q\otimes r}}
    	\arrow[from=2-1, to=2-2]
    	\arrow[from=1-1, to=2-1]
    	\arrow[from=1-1, to=1-2]
    	\arrow[from=1-2, to=2-2]
    \end{tikzcd}\]
    It suffices to show that the following diagram commutes.
    \[\begin{tikzcd}
    	{\cofree_p \otimes \cofree_q \otimes r} & {\cofree_p \otimes q \otimes r} & {\cofree_p \otimes \free_{q \otimes r}} \\
    	{\cofree_{p \otimes q} \otimes r} & {p \otimes q \otimes r} & {\free_{p \otimes q \otimes r}}
    	\arrow["{\modstruct_{p, q \otimes r}}", from=1-3, to=2-3]
    	\arrow[from=2-2, to=2-3]
    	\arrow[from=1-2, to=2-2]
    	\arrow[from=1-2, to=1-3]
    	\arrow[from=1-1, to=1-2]
    	\arrow[from=2-1, to=2-2]
    	\arrow[from=1-1, to=2-1]
    \end{tikzcd}\]
    The square on the left commutes by definition of the laxator of $\cofree_-$ while the square on the right commutes by definition of $\modstruct_{p, q \otimes r}$.
\end{proof}

\end{document}